\documentclass[12pt,reqno]{amsart}
\usepackage{amsthm}

\usepackage{amssymb}
\usepackage{graphics}
\usepackage{tikz}
\usetikzlibrary{shapes,backgrounds,calc}
\usepackage{latexsym}
\usepackage{multicol}
\usepackage{verbatim,enumerate}
\usepackage{accents}
\usepackage{cite}
\usepackage{enumitem}
\usepackage{array}
\usepackage{mathtools}

\usepackage[colorlinks=true, linkcolor=blue, citecolor=blue, urlcolor=blue]{hyperref}
\usepackage{hyperref}
\usepackage{amsmath, amscd,url}

\usepackage{setspace}
\usepackage{pstricks}

\advance\textwidth by 1.5in 
\usepackage{geometry}
\theoremstyle{definition}
\newtheorem{cor}{Corollary}
\newtheorem{lem}{Lemma}

\newtheorem{thm}{Theorem}

\theoremstyle{definition}
\newtheorem{defn}{Definition}

\theoremstyle{definition}
\newtheorem{example}{Example}
\newtheorem{rem}{Remark}

\newenvironment{pf}{\proof}{\endproof}
\newcounter{cnt}
 \makeatletter
\def\mydggeometry{\makeatletter\dg@YGRID=1\dg@XGRID=20\unitlength=0.003pt\makeatother}
\makeatother \theoremstyle{remark}

\numberwithin{equation}{section}
\let\bwdg\bigwedge
\def\bigwedge{{\textstyle\bwdg}}

\newcommand{\nc}{\newcommand}
\newcommand{\rnc}{\renewcommand}

\nc{\cal}{\mathcal} \nc{\goth}{\mathfrak} \rnc{\bold}{\mathbf}

\nc\bomega{{\mbox{\boldmath $\omega$}}} \nc\bpsi{{\mbox{\boldmath $\Psi$}}}
 \nc\balpha{{\mbox{\boldmath $\alpha$}}}
 \nc\bpi{{\mbox{\boldmath $\pi$}}}
 \nc\bvpi{{\mbox{\boldmath $\varpi$}}}
\nc\chara{\operatorname{ch}}

  \nc\bxi{{\mbox{\boldmath $\xi$}}}
\nc\bmu{{\mbox{\boldmath $\mu$}}} \nc\bcN{{\mbox{\boldmath $\cal{N}$}}} \nc\bcm{{\mbox{\boldmath $\cal{M}$}}} \nc\blambda{{\mbox{\boldmath
$\lambda$}}}\nc\bnu{{\mbox{\boldmath $\nu$}}}

\makeatletter
\def\section{\def\@secnumfont{\mdseries}\@startsection{section}{1}%
  \z@{.7\linespacing\@plus\linespacing}{.5\linespacing}%
  {\normalfont\scshape\centering}}
\def\subsection{\def\@secnumfont{\bfseries}\@startsection{subsection}{2}%
  {\parindent}{.5\linespacing\@plus.7\linespacing}{-.5em}%
  {\normalfont\bfseries}}
\makeatother

 \nc{\Hom}{\operatorname{Hom}}
  \nc{\mode}{\operatorname{mod}}
\nc{\End}{\operatorname{End}} \nc{\wh}[1]{\widehat{#1}} \nc{\Ext}{\operatorname{Ext}} \nc{\ch}{\text{ch}} \nc{\ev}{\operatorname{ev}}
\nc{\Ob}{\operatorname{Ob}} \nc{\soc}{\operatorname{soc}} \nc{\rad}{\operatorname{rad}} \nc{\head}{\operatorname{head}}

\def\mult{\operatorname{mult}}

 \nc{\Cal}{\cal} \nc{\Xp}[1]{X^+(#1)} \nc{\Xm}[1]{X^-(#1)}
\nc{\on}{\operatorname} \nc{\Z}{{\bold Z}} \nc{\J}{{\cal J}}  \nc{\Q}{{\bold Q}}

\nc{\N}{{\bold N}}  \nc\boa{\bold a} \nc\bob{\bold b} \nc\boc{\bold c} \nc\bod{\bold d} \nc\boe{\bold e} \nc\bof{\bold f} \nc\bog{\bold g}
\nc\boh{\bold h} \nc\boi{\bold i} \nc\boj{\bold j} \nc\bok{\bold k} \nc\bol{\bold l} \nc\bom{\bold m} \nc\bon{\mathbb n} \nc\boo{\bold o}
\nc\bop{\bold p} \nc\boq{\bold q} \nc\bor{\bold r} \nc\bos{\bold s} \nc\boT{\bold t} \nc\boF{\bold F} \nc\bou{\bold u} \nc\bov{\bold v}
\nc\bow{\bold w} \nc\boz{\bold z}\nc\ba{\bold A} \nc\bb{\bold B} \nc\bc{\mathbb C} \nc\bd{\bold D} \nc\be{\bold E} \nc\bg{\bold
G} \nc\bh{\bold H} \nc\bi{\bold I} \nc\bj{\bold J} \nc\bk{\bold K} \nc\bl{\bold L} \nc\bm{\bold M} \nc\bn{\mathbb N} \nc\bo{\bold O} \nc\bp{\bold
P} \nc\bq{\bold Q} \nc\br{\bold R} \nc\bs{\bold S} \nc\bt{\bold T} \nc\bu{\bold U} \nc\bv{\bold V} \nc\bw{\bold W} \nc\bz{\mathbb Z} \nc\bx{\bold
x} \nc\KR{\bold{KR}} \nc\rk{\bold{rk}} \nc\het{\text{ht }}

\nc\toa{\tilde a} \nc\tob{\tilde b} \nc\toc{\tilde c} \nc\tod{\tilde d} \nc\toe{\tilde e} \nc\tof{\tilde f} \nc\tog{\tilde g} \nc\toh{\tilde h}
\nc\toi{\tilde i} \nc\toj{\tilde j} \nc\tok{\tilde k} \nc\tol{\tilde l} \nc\tom{\tilde m} \nc\ton{\tilde n} \nc\too{\tilde o} \nc\toq{\tilde q}
\nc\tor{\tilde r} \nc\tos{\tilde s} \nc\toT{\tilde t} \nc\tou{\tilde u} \nc\tov{\tilde v} \nc\tow{\tilde w} \nc\toz{\tilde z} \nc\woi{w_{\omega_i}}

\newcommand\wrapped[1]%
{%
	\begin{array}{@{}l@{}}#1\end{array}%
}

\ifTUTeX
  \usepackage{fontspec}
\else
  \usepackage[T1]{fontenc}
  \usepackage[utf8]{inputenc} 
  \DeclareUnicodeCharacter{200B}{{\hskip 0pt}}
\fi

\begin{document}
	
	\title{On the spectrum of generalized H--join operation constrained by indexing maps - I}
		
	\author{R. GaneshBabu}
	\address{Indian Institute of Technology Madras, Chennai, India.}
	\email{ma22d011@smail.iitm.ac.in}
	\author{G. Arunkumar$^{\ast}$}
	\thanks{}
	\address{Indian Institute of Technology Madras, Chennai, India.}
	\email{garunkumar@iitm.ac.in}


		\thanks{$^{\ast}$-The corresponding author. He acknowledges the NFIG grant of the Indian Institute of Technology Madras RF/22-23/0985/MA/NFIG/009003 and the SERB Startup Research Grant SRG/2022/001281. The first author acknowledges the financial support from Prime Minister's Research Fellowship (ID: 2503528)}. 
	
	
	\maketitle
 
 \begin{abstract}
Fix $m \in \mathbb N$. A new generalization of the $H$-join operation of a family of graphs $\{G_1, G_2, \dots, G_k\}$ constrained by indexing maps $I_1,I_2,\dots,I_k$ is introduced as $H_m$-join of graphs, where the maps $I_i:V(G_i)$ to $[m]$. Various spectra, including adjacency, Laplacian, and signless Laplacian spectra, of any graph $G$, which is a $H_m$-join of graphs is obtained by introducing the concept of $E$-main eigenvalues. More precisely, we deduce that in the case of adjacency spectra, there is an associated matrix $E_i$ of the graph $G_i$ such that a $E_i$-non-main eigenvalue of multiplicity $m_i$ of $A(G_i)$ carry forward as an eigenvalue for $A(G)$ with the same multiplicity $m_i$, while an $E_i$-main eigenvalue of multiplicity $m_i$ carry forward as an eigenvalue of $G$ with multiplicity at least $m_i - m$. As a corollary, the universal adjacency spectra of some families of graphs is obtained by realizing them as $H_m$-joins of graphs. As an application, infinite families of cospectral families of graphs are found. 
\end{abstract}
	
 	\section{Introduction}\label{intro}
All the graphs considered in this paper are finite and simple. Let $G$ be a graph on $n$ vertices. Let $V(G)$, $E(G)$, $A(G)$, and $D(G)$ be the vertex set, the edge set, the adjacency matrix, and the degree matrix associated with the graph $G$. Let $I_n$ and $J_n$ be the identity and all-one matrix of order $n$, respectively. The universal adjacency matrix of $G$, denoted by $U(G)$, is  the matrix $U(G) = \alpha A(G) + \beta I_n + \gamma J_n + \delta D(G)$ where $\alpha,\beta,\gamma,\delta \in \mathbb R$ and $\alpha \ne 0$. By the $X$-spectra of a graph $G$, we mean the multiset of eigenvalues of $X(G)$, where $X(G)$ is a matrix associated with $G$ such as the adjacency matrix $A(G)$, the Laplacian matrix $L(G)=D(G)-A(G)$, the signless Laplacian matrix $Q(G)=D(G)+A(G)$, the Seidal matrix $S(G)=J_n-I_n-2A(G)$ and the more general universal adjacency matrix $U(G)$. We denote the $X$-spectra of $G$ by $\sigma(X(G))$.

The study of spectra of graphs obtained from various graph operations is a well-explored area. See \cite{Barik} for a survey on this topic. 
 Let $H$ be a graph on the vertex set $V(H)=\{v_1,v_2,\dots,v_k\}$ and $\mathcal{F}=\{G_1,G_2,\dots,G_k\}$ be a family of graphs with $V(G_i) = \{v_i^{1},\dots,v_i^{n_i}\}$ for $1\leq i \leq k.$ In this paper, we mainly deal with graph operations in which the vertex $v_i$ of the graph $H$ is replaced by the graph $G_i$, and depending on the adjacency between vertices $v_i$ and $v_j$ in $H$, we make adjacency between vertices of $G_i$ and those of $G_j$ in a variety of ways. One main object of interest when it comes to these operations is to establish a connection between the $X$-spectra of the resultant graph $G$ and the $X$-spectra of the factor graphs $\{G_1,G_2,\dots,G_k\}$.
 
 In this direction, numerous results have been proved in the literature concerning the eigenvalues of $G_i$, which get carried forward to the resultant graph $G$, especially in the case of the $H$-join operation:  The $H$-join of $G_1,G_2,\dots,G_k$ is the graph obtained by replacing the vertex $v_i$ of $H$ by the graph $G_i$ (for each $i$) and by letting each vertex of $G_i$ be adjacent to each vertex of $G_j$ if $v_i$ and $v_j$ are adjacent in $H$. This operation was introduced by Schwenk as the generalized composition of graphs in \cite{sch} and was reintroduced as the $H$-join of graphs by Cardoso et al. in \cite{hjn}. In \cite{cdo11}, Cardoso et al. obtained the $A$-spectra of $H$-join of a family of graphs $\mathcal{F}$ when each $G_i$ is a regular graph and $H$ is the path graph $P_k$. Subsequently, in \cite{hjn}, Cardoso et al. obtained the $A$-spectra of $H$-join of the family $\mathcal{F}$  when each $G_i$ is a regular graph and $H$ is any graph. Recently, in \cite[Theorem 4]{msg}, Saravanan et al. proved the following result about the $U$-spectra of arbitrary $H$-joins of graphs, i.e., for any $H$ and any $G_i$ (not necessarily regular).
Let $G$ be the $H$-join of the family of graphs $\mathcal{F}=\{G_1,G_2,\dots,G_k\}.$ Then
$$\det(\lambda I_{n} - U(G)) =\Bigg( \displaystyle \Pi_{i=1}^k \phi_i(\lambda-\delta w_i) \Gamma_i
\Bigg)\det(\widetilde{U}(G))$$

\begin{equation}\label{UAdofG}
	\text{where}\,\,    \widetilde{U}(G) =   \begin{bmatrix}
	\frac{1}{\Gamma_1} & -(\rho_{1,2} \alpha +\gamma)  & \cdots & -(\rho_{1,k} \alpha +\gamma)  \\
	-(\rho_{2,1} \alpha +\gamma)  & \frac{1}{\Gamma_2} & \cdots & -(\rho_{2,k} \alpha +\gamma) \\
	\vdots & \vdots & \ddots & \vdots  \\
	-(\rho_{k,1} \alpha +\gamma) & -(\rho_{k,2} \alpha +\gamma) & \cdots &\frac{1}{\Gamma_{k}}
	\end{bmatrix}
\end{equation}	
where $n=\sum_{i=i}^{k} n_i, w_i = \displaystyle \sum_{v_l \in N_H(v_i)}n_l, \phi_i(\lambda-\delta w_i)=\det((\lambda-\delta w_i) I_{n_i}-U(G_i)$), $ \Gamma_{(U(G_i)+\delta w_i I_{n_i})}(\bold 1_{n_i}) = \Gamma_i:=\bold 1_{n_i}^t (\lambda I_{n_i}-(U(G_i)+\delta w_i i_{n_i}))^{-1} \bold 1_{n_i}.$

The functions $\Gamma_i$ are called the main functions associated with the graph $G_i$. Thus,
the authors have expressed the characteristic polynomial of the universal adjacency matrix of $H$-join of arbitrary graphs in terms of the characteristic polynomials of the universal adjacency matrices and the main functions of the factor graphs $\{G_1,\dots, G_k\}$. As a consequence of Equation \eqref{UAdofG}, they also obtained the following. For every non-main eigenvalue $\lambda$ of $U(G_i)$, $\lambda+\delta w_i$  is an eigenvalue of $U(G)$ with the same multiplicity, whereas
for every main eigenvalue $\lambda$ of $U(G_i)$, $\lambda+\delta w_i$  is an eigenvalue of $U(G)$ with possibly one less multiplicity. This is achieved by introducing the concept of $u$-main eigenvalue for an arbitrary vector $u$: An eigenvalue $\lambda$ of the matrix $X(G)$ is called $u$-main if its eigenspace is not orthogonal to the vector $u$. This notion is a generalization of main eigenvalues introduced by Cvetkovic in \cite{cvetmain}, in which case, we take $u=\bold 1_{n}$ the all-one vector of size $n \times 1$. See \cite{mainsurvey} for a survey on the study of main eigenvalues. In \cite{cons}, the \textit{$H$-generalized join} of graphs with respect to $\mathcal{S}=\{S_i : S_i \subseteq V(G_i), i \in [k]\}$ is introduced and obtained by replacing the vertex $v_i$ by the graph $G_i$ for each $i$ and by letting each vertex of $S_i$ be adjacent to each vertex of $S_j$ if $v_i$ and $v_j$ are adjacent in $H$. The Spectra of this operation is studied in \cite{msg,cons}.

In this paper, we introduce a new graph operation that generalizes $H$-generalized join of graphs (in turn generalizes $H$-join of graphs), which we call the $H_m$-join of graphs: Fix $m \in \mathbb{N}$. For $1 \le i \le k$, let $I_i:V(G_i) \to [m]:=\{1,\dots,m\}$ be arbitrary (indexing) maps. The $H_m$-join of a family of graphs $\mathcal{F}=\{G_1,G_2,\dots,G_k\}$ with respect to the indexing maps $\mathcal{I}=\{I_1,I_2,\dots,I_k\}$ is obtained by replacing the vertex $v_i \in V(H)$ by the graph $G_i$ (for each $i$) and by letting a vertex $v_i^p\in V(G_i)$ be adjacent to a vertex  $v_j^q \in V(G_j)$ if  $v_i$ and $v_j$ are adjacent in $H$ and $I_i(v_i^p)=I_j(v_j^q)$. The Cartesian product $G_1\openbox G_2$ of any two graphs, Generalized Petersen graphs, Generalized web graphs, Generalized helm graphs, Lollipop graphs, and tadpole graphs are some examples of graph families which can be realized as $H_m$-joins [c.f. Section \ref{Otherfamilies}]. This shows that the spectra of a wide range of graphs can be studied uniformly using our $H_m$-join operation.

We define the $n_i \times m$ \textit{indexing} matrix $E_i$ associated with the pair $(G_i,I_i)$ as follows:   $
(E_i)_{st}=\begin{cases}
    1 & \text{if}\ I_i(v_i^s)=t \\
    0 & \text{otherwise.}
\end{cases}
$
 These matrices play an important role in proving the results concerning the study of spectra of graphs obtained from the $H_m$-join operation of graphs. 

In Section \ref{basicterms}, we shall see  that 
the concept of $u$-main eigenvalues is insufficient to study the spectra of graphs resulting from the $H_m$-join operation. Consequently, we generalize $u$-main eigenvalues to $E$-main eigenvalues, where $E$ is an arbitrary rectangular matrix. An eigenvalue $\lambda$ of the matrix $X(G)$ is defined as an $E$-main eigenvalue if its eigenspace is not orthogonal to the column space of the matrix $E$. 

In Section \ref{mainresult}, similar to Equation \eqref{UAdofG}, we discuss an analogous description of the characteristic polynomial of $U(G)$ of the graph $G$ which is obtained from the $H_m$-join operation. We also prove that, for each $1 \le i \le k$, every $E_i$-non-main eigenvalue of $A(G_i)$ gets carried forward to $G$ with the same multiplicity and $E_i$-main eigenvalues carry forward at least $n_i-m$ times.

In Section \ref{hgeneralizedjoin}, we see applications of our findings in calculating the spectra of graphs obtained by the $H$-generalized join operation. We first realize an arbitrary $H$-generalized join of graphs as an $H_m$-join of graphs as follows: Let $m \ \text{be}\  k+1$, where $k$ is the number of vertices in $H$. Then, for each $i\in [k]$, define
$I_i(v_i^p)=\begin{cases}
    1 & \text{if}\ v_i^p\in S_i \\
    i+1 & \text{otherwise}
\end{cases}.$ Using this realization, we compute the $U$-spectra of an arbitrary $H$-generalized join of graphs.
 Finally, we find infinite families of non-isomorphic $H$-generalized joins of graphs, which are $A$-cospectral, $S$-cospectral, and $L$-cospectral.

\section{Some basic definitions and lemmas}\label{basicterms}

Given a graph $G$, we denote its vertex and edge set by $V(G)$ and $E(G)$, respectively. If there is an edge between vertices, say $i$ and $j$ of $G$, then we write $ij \in E(G)$. Also, for simplicity, we denote the adjacency spectrum of the graph $G$ by $\sigma(A(G))$. For a matrix $M$, we shall denote its $ij$-th entry by $(M)_{ij}$ and for an eigenvalue $\lambda$ of $M$, $\xi_M(\lambda)$ shall denote its eigenspace. For basic notations and terminologies in graph theory, we follow \cite{diestel}.
 
We use these notations throughout.

\subsection{$H_m$-join of graphs }
We start with the formal definition of $H$-join of graphs. 
 \begin{defn}\label{h join}
 Let $H$ be a graph on the vertex set $V(H) =\{v_1,v_2,\dots,v_k\}$ with the adjacency matrix ${A}(H) = (a_{ij})_{1\leq i,j \leq k}$. Let $\mathcal{F} = \{G_i \colon 1\leq i \leq k\}$ be an arbitrary family of graphs with vertex sets $V(G_i) = \{v_i^{1},\dots,v_i^{n_i}\}$ for $1\leq i \leq k$.  The {\em{H}-join} of the family $\mathcal{F}$ of graphs, denoted by $\bigvee_H^{\mathcal{F}}$, is obtained by replacing each vertex $v_i$ of $H$ by the graph $G_i$ in $H$ such that if $v_i$ is adjacent to $v_j$ in $H$ then each vertex of $G_i$ is adjacent to every vertex of $G_j$ in $\bigvee_H^{\mathcal{F}}$. More precisely, 
 \begin{enumerate}
     \item $V\big( \bigvee_{H}^\mathcal{F} \big) = \displaystyle \bigsqcup_{i=1}^k V(G_i)$, and
     \item $E\big( \bigvee_{H}^\mathcal{F} \big) = \big(\displaystyle  \bigsqcup_{i=1}^k E(G_i)\big) \sqcup \big(   \bigsqcup_{ (v_i,v_j) \in E(H)} \{ xy : x \in V(G_i), y \in V(G_j) \}\big).$
 \end{enumerate}
 \end{defn}

Next, we introduce the $H_m$-join of graphs motivated by the above-defined notion of $H$-join of graphs.
 
\begin{defn} Let the graph $H$, and the collection of graphs $\mathcal F= \{G_i \colon 1\leq i \leq k\}$ be as in the previous definition. 
Fix $m \in \mathbb N$, and let $\mathcal{I} = \{I_i: 1 \le i \le k\}$ be a collection of maps $I_i: V(G_i) \to [m]$ which we call the indexing maps. Then, the $H_m$-join of the graphs $\{G_i \colon 1\leq i \leq k\}$, denoted by $\bigvee_{H}^{\mathcal{F,I}}$, is obtained by replacing the $i$th vertex of the graph $H$ by the graph $G_i$ and, if the vertices $i$ and $j$ of $H$ are adjacent then the pair of vertices $\{u,v\}$ where $u \in V(G_i)$ and $v \in V(G_j)$ satisfying $I_i(u) = I_j(v)$ are adjacent in $\bigvee_{H}^{\mathcal{F,I}}$. Thus we have
\begin{enumerate}
			\item $V(\bigvee_{H}^{\mathcal{F,I}}) = \displaystyle{\bigsqcup_{i=1}^k V(G_i)}$, and
            \item $E(\bigvee_{H}^{\mathcal{F,I}}) =  \displaystyle{(\bigsqcup_{i=1}^k E(G_i))} \sqcup(\bigsqcup_{(i,j)\in E(H)} \{uv: u \in V(G_i), v\in V(G_j) \text{ with }I_i(u)=I_j(v)\})  .$
\end{enumerate}
\end{defn}

\begin{example}
The following graph $G$ is the $(P_4)_5$-join of $\mathcal F=\{K_3,P_4,C_5,K_{3,3}\}$ with the indexing functions $I_1,I_2,I_3,I_4$ and each vertex $v$ is labeled with the values of these indexing functions.
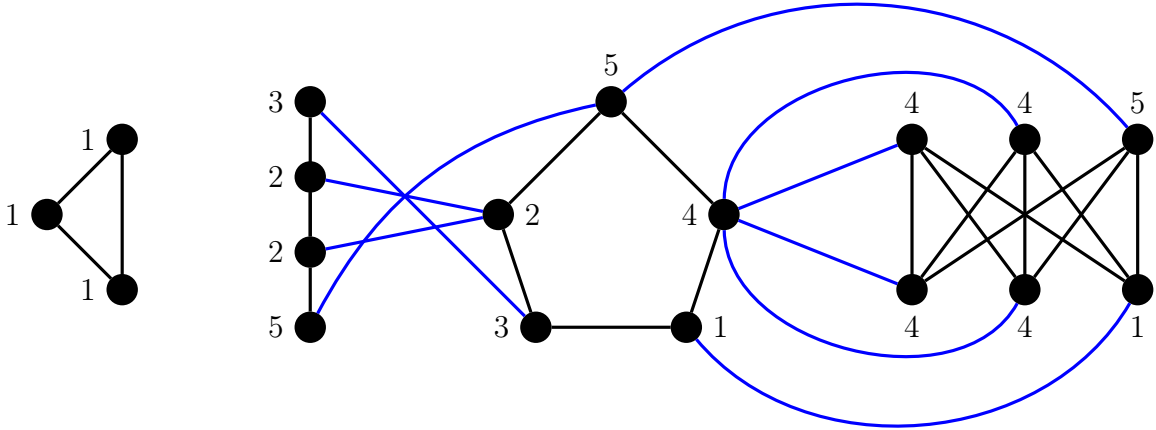
\begin{figure}[h]
\centering
\begin{tikzpicture}[
every edge/.style = {draw=black,very thick},
 vrtx/.style args = {#1/#2}{%
      circle, draw, thick, fill=black,
      minimum size, label=#1:#2}
                    ]
\node(A) [vrtx=left/1] at (0, 1.5) {};
\node(B) [vrtx=left/1] at (-1, 0.5) {};
\node(C) [vrtx=left/1] at (0,-.5) {};
\path   (A) edge (B)
        (A) edge (C)
        (B) edge (C);

\node (2A) [vrtx=left/2]     at ( 2.5, 0) {};
\node (2B) [vrtx=left/2]     at (2.5, 1) {};
\node (2C) [vrtx=left/3]    at ( 2.5,2) {};
\node (2D) [vrtx=left/5]    at ( 2.5,-1) {};
\path   (2A) edge (2B)
        (2A) edge (2C)
        (2A) edge (2D);

\node (3A) [vrtx=above/5]     at ( 6.5, 2) {};
\node (3B) [vrtx=right/2]     at (5, 0.5) {};
\node (3C) [vrtx=left/4]    at ( 8,0.5) {};
\node (3D) [vrtx=right/1]    at ( 7.5,-1) {};
\node (3E) [vrtx=left/3]    at ( 5.5,-1) {};
\path   (3A) edge (3B) 
(3B) edge (3E)
(3D) edge (3C)
(3E) edge (3D)
(3C) edge (3A)
(2B) edge[blue] (3B)
       (2A) edge[blue] (3B)
       (2C) edge[blue] (3E)
       (2D) edge[blue,bend left=25] (3A);
\node (4A) [vrtx=above/4]     at ( 10.5, 1.5) {};
\node (4B) [vrtx=above/4]     at (12, 1.5) {};
\node (4C) [vrtx=above/5]    at ( 13.5,1.5) {};
\node (4D) [vrtx=below/4]    at ( 10.5,-.5) {};
\node (4E) [vrtx=below/4]    at ( 12,-.5) {};
\node (4F) [vrtx=below/1]    at ( 13.5,-.5) {};

\path (4A) edge  (4D) edge (4E) edge (4F)
(4B) edge  (4D) edge (4E) edge (4F)
(4C) edge  (4D) edge (4E) edge (4F)

(3C) edge[blue] (4A) edge[blue] (4D) edge[blue,bend left=75] (4B) edge[blue,bend right=75] (4E) 
(4F) edge[blue, bend left=55](3D)

(3A) edge[blue,bend left=45] (4C);  
\end{tikzpicture}
\caption{$(P_4)_5$-join of $K_3,P_4,C_5,K_{3,3}$}
    \label{fig:figure-3}
\end{figure}

\end{example}

    In the above diagram, the newly added edges between the graphs $G_i$s are given in blue. Note that $G$ is not the $P_4$-join of  $\{K_3, P_4, C_5,K_{3,3}\}$ since, for instance, no vertex of $G_1=K_3$ is adjacent to a vertex of $G_2=P_4$ although the underlying vertices in $H=P_4$ are adjacent. We observe that the above-given graph is also a $(K_1 \sqcup P_3 )_5$-join of $\mathcal F=\{K_3, P_4, C_5,K_{3,3}\}$ with the same set of indexing maps.

    Note that if each $I_j$ for $1 \le j \le k$ is the constant function $c$ for a fixed $1\le c \le m$ , then the operation $H_m$-join of graphs coincides with the $H$-join of graphs. In particular, this is the case when $m=1$.  In a different ongoing project, a graph operation called the free-H join of graphs (where the edges between the graphs $G_1,\dots,G_k$ can be made completely arbitrarily) is being studied \cite{agg}.
\subsection{Various graph families as $H_m$-joins.}\label{Otherfamilies}
    
In this section, we will see a list of graph families which can be realized as $H_m$-joins of graphs.

\begin{enumerate}[leftmargin=*]
    \item The Cartesian product of any two graphs $A_1\times A_2$ are $H_m$-joins.
    
    Let $A_1$ and $A_2$ be graphs with $V(A_1)=\{u_1, u_2, \dots, u_{n_1} \}$ and $V(A_2)=\{v_1, v_2, \dots, v_{n_2}\}$. 
    Let $G$ be the $(A_1)_{n_2}$-join of $\{G_1, G_2, \dots, G_{n_1} \}$ with $G_1=G_2=\dots=G_{n_1}=A_2$ with $V(G_i)=\{v_i^1,v_i^2,\dots,v_i^{n_2}\}$ for each $i\in[n_1]$ and indexing maps $I_1, I_2, \dots, I_{n_1}$ defined by  $I_i(v_i^j)=j$ for each $i\in [n_1], \ j \in [n_2]$ and .  
    Identifying the vertex $v_i^j$ of $G$ with $(u_i,v_j)$ of $A_1 \times A_2$, we see that $G\cong A_1 \times A_2$.   Similarly, $A_1 \times A_2$ is $(A_2)_{n_1}$-join of $n_2$ number of $G_1$.

We emphasize that the Ladder graphs $P_n \times P_2$, more generally the Planar Grids $P_n \times P_m$, Prisms $C_n \times C_m$, Books $S_n \times P_2$, where $S_n$ is the star graph on $n$ vertices, Stacked books $S_n \times P_m$, the hypercube $Q_n$, Generalized books $S_{2n} \times Q_m$ and Generalized ladders $P_{2m+1}\times Q_n$ are some important graph families which arise as the  Cartesian product of two graphs.

\item \textit{Generalized Petersen graph $P(n,k)$}\cite{peter}

The generalized Petersen graph $P(n,k)$, for $n \geq 5 \text{ and } k < \frac{n}{2}$, is the graph with the vertex set $V(P(n,k))=\{ a_i,b_i : 0 \leq i \leq n-1\}$, and edge set $E(P(n,k))=\{a_ia_{i+1}, a_ib_i, b_ib_{i+k} : 0 \leq i \leq n-1\}$,
where the subscripts are expressed as integers modulo $n$. 
Let $AP(n,k)$ be the induced subgraph of $P(n,k)$ generated by the vertices $\{a_i : 0 \leq i \leq n-1\}$ and $BP(n,k)$ be the induced subgraph generated by the vertices $\{b_i : 0 \leq i \leq n-1\}$. By \cite[Lemma 2.1]{peter}, $AP(n,k) \cong C_n$ and $BP(n,k) \cong d C_{\frac{n}{d}}$, the disjoint union of $d$ many cycle graph $C_{\frac{n}{d}}$, where $d=\gcd(n,k)$.  From this information, we observe that $P(n,k)$ is the $(K_2)_n$-join of $C_n$ and $dC_\frac{n}{d}$ with indexing maps $I_1(a_i)=I_2(b_i)=i+1$ for each  $0 \leq i \leq n-1$.

\item \textit{Generalized helm graph $H_n^m$.}\cite{gallian}

The Wheel graph $W_n$ on $n+1$ vertices is the $K_2$-join of $C_n$ and $K_1$. A Helm graph $H_n$ is obtained from the Wheel graph $W_n$ by attaching a pendant edge to each vertex of the cycle $C_n$ in $W_n$. The Generalized helm graph $H_n^m$ can be obtained by attaching the path graph $P_{m+1}$ on each vertex of $C_n$ in $W_n$  by a bridge. It is clear that the graph $H_n^m$ is a $(K_2)_3$-join of $W_n$ and $nP_{m+1}$, where the vertices of $C_n$ in $W_n$ and the $n$ pendant vertices of the $n$ copies of $P_{m+1}$ which are attached to $W_n$ are indexed by $1$, the remaining vertex in $W_n$ is indexed by $2$ and the remaining vertices in $nP_{m+1}$ are indexed by $3$.

\item \textit{Genaralized web graph $W(t,n)$.}\cite{gallian}

 A web graph $W(2,n)$ is obtained from the Helm graph $H_n$ by joining the pendant vertices of the Helm graph to form a cycle and attaching a pendant edge at each vertex of this outer cycle. The generalized web graph $W(t,n)$ is obtained by iterating $t$ times this process of joining the pendant vertices to form a cycle and attaching a pendant edge at each vertex of this outer cycle. Note that $\lvert V(W(t,n)\rvert = (t+1)n+1.$ 
Let $V(W_n)=\{v_1,v_2,\dots,v_n,c\}$ where $c$ is the vertex of $K_1$ in $W_n$. Consider $t$ number of copies of $C_n$: $C^1,\dots,C^t$ with $V(C^i)=\{v_i^1,\dots,v_i^n\}$ for $i\in[t]$. Consider also a copy of $\overline{K_n}$ with $V(\overline{K_n})=\{v_{t+1}^1,v_{t+1}^2,\dots,v_{t+1}^n\}$. 
Now define $I_1(c)=1, I_1(v_i)=i+1$ for $i\in[n]$ and define $I_j(v_j^k)=k+1$ for $2 \leq j \leq t+2,1 \leq k \leq n $. Now, we see that the graph $W(t,n)$ is a $(P_{t+2})_{n+1}$-join of the graphs $W_n,C^1\cong C_n,\dots,C^t\cong C_n,\overline{K_n}$ with respect to the maps $I_1, I_2,\dots,I_{t+2}$.

\item \textit{Lollipop graph $L_{m,n}$ \cite{lollipop} and Tadpole graph $T_{m,n}$}\cite{tadpole}.

The lollipop graph with parameters $(m,n)$ is obtained by considering a complete graph $K_m$ and a path graph $P_n$ and connecting a pendant vertex $u$ of $P_n$ to any vertex $v$ of $K_m$. Clearly, $L_{m,n}$ is a $(K_2)_3$-join of $K_m$ and $P_n$ with $I_1(V(K_m)-\{v\})=\{1\}, I_1(v)=2=I_2(u), I_2(V(P_n)-\{u\})=\{3\}$. The Tadpole graph $T_{m,n}$ is defined in a similar way to the Lollipop graph $L_{m,n}$ where we consider the cycle graph $C_m$ in place of the complete graph $K_m$. 

\end{enumerate}
    In what follows, we study the universal spectrum of graphs obtained by this $H_m$-join operation. As a first step, in the next subsection, we prove some basic lemmas that will be helpful for the rest of the paper.

\subsection{Adjacency matrix of $H$-join and $H_m$-join of graphs} 

 Consider a graph $H$ on vertices $\{v_1,v_2,\dots,v_k\}$ and a family of graphs $\mathcal{F}=\{G_1,G_2, \hdots, G_k\}$ with $V(G_i)=\{v_i^1,v_i^2,\hdots,v_i^{n_i} \}$. Let $A_i$ be the adjacency matrix of $G_i$ and $\rho_{i,j}=\begin{cases}
    1 & \text{if} \ v_iv_j\in E(H),\\
    0 & \text{otherwise}
\end{cases}$. Then, the adjacency matrix of the $H$-join $\bigvee_H^{\mathcal{F}}$ is given by  
 \begin{equation} \label{AofG}
 A(\bigvee_H^{\mathcal{F}}) = \begin{bmatrix}
A_1  & \rho_{1,2} \textbf{1}_{n_1}\textbf{1}_{n_2}^t & \cdots &  \rho_{1,k}\textbf{1}_{n_1} \textbf{1}_{n_k}^t \\ \rho_{2,1}
\textbf{1}_{n_2} \textbf{1}_{n_1}^t & A_2 & \cdots & \rho_{2,k}\textbf{1}_{n_2} \textbf{1}_{n_k}^t \\
\vdots & \vdots & \ddots & \vdots  \\ \rho_{k,1}
\textbf{1}_{n_k} \textbf{1}_{n_1}^t &\rho_{k,2}\textbf{1}_{n_k} \textbf{1}_{n_2}^t & \cdots & A_k
\end{bmatrix}.
  \end{equation}

In \cite[Theorem 2]{msg}, the following result discussing the adjacency spectrum of the above matrix is proven.

 \begin{thm}\label{premain}
	Let $M_i$ be a complex matrix of order $n_i$ and let $u_i$ and $v_i$ be arbitrary complex vectors of size $n_i \times 1$ for $1 \le i \le k$. Let $n=\displaystyle \Sigma_{i=1}^k n_i.$ Let $\rho_{i,j}$ be arbitrary complex numbers for $1 \le i,j \le k$ and $i \ne j$. For each $1 \le i \le k$, let $\phi_i(\lambda)=\det(\lambda I_{n_i}-M_i)$ be the characteristic polynomial of the matrix $M_i$ and $\Gamma_i= \Gamma_{M_i}(u_i,v_i) = v_i^t (\lambda I - M_i)^{-1} u_i$. 
	Let $\bold M$ be the $k$-tuple $(M_1, M_2, \dots, M_k)$, $\bold u$ be the 2$k$-tuple $(u_1,v_1, u_2,v_2 \dots, u_k,v_k)$ and $\rho$ be the ${k(k-1)}$-tuple $ (\rho_{1,2}, \rho_{1,3} \dots, \rho_{1,k},\rho_{2,1}, \rho_{23}, \dots, \rho_{2,k}, \dots, \rho_{k,1}, \rho_{k,2}, \dots,  \rho_{k-1 k})$.
	 Considering $\bold M$, $\bold u$ and $\rho$, the following matrices are defined: $$A(\bold M, \bold u, \rho) := \begin{bmatrix}
	
	M_1 & \rho_{1,2} u_1v_2^t & \cdots  & \rho_{1,k} u_1v_k^t \\
	\rho_{2,1} u_2v_1^t & M_2 & \cdots  & \rho_{2,k} u_2v_k^t \\
	\vdots & \vdots & \ddots  & \vdots \\
	\rho_{k,1} u_kv_1^t & \rho_{k,2} u_kv_2^t & \cdots  &M_k
\end{bmatrix}$$ $$\text{ and }
 \widetilde{A}(\bold M, \bold u, \rho) :=  \begin{bmatrix}
1 & -\rho_{1,2}\Gamma_1 & \cdots & -\rho_{1,k}\Gamma_1 \\
-\rho_{2,1}\Gamma_2  & 1 & \cdots & -\rho_{2,k}\Gamma_2 \\
\vdots & \vdots & \ddots & \vdots  \\
-\rho_{k,1}\Gamma_{k}& -\rho_{k,2}\Gamma_{k}& \cdots &1
\end{bmatrix}.$$

Then the characteristic polynomial of $A(\bold M, \bold u, \rho)$ is given by
\begin{equation}\label{maineqn}
\det(\lambda I_n - A(\bold M, \bold u, \rho)) = \Bigg( \Pi_{i=1}^k \phi_i(\lambda)\Bigg) \det(\widetilde{A}(\bold M, \bold u, \rho)) .
\end{equation} 
\end{thm}

Let $G$ be the $H_m$-join of the collection of graphs $\mathcal F = \{G_1, G_2,\hdots,G_k \}$ with indexing maps $\{I_i: 1 \le i \le k\}$ and $V(G_i)=\{v_i^1,v_i^2,\hdots,v_i^{n_i} \}$. Then with respect to the ordering $\{v_1^1,v_1^2,\hdots,v_1^{n_1},v_2^1,\hdots,v_2^{n_2},\hdots,v_k^1,\dots,v_k^{n_k}\}$, let the adjacency matrix of $G$ be as follows.

\begin{equation} \label{gen form}
    A(G)=\begin{bmatrix}
    A(G_1) &  B_{12} & \cdots  &  B_{1k} \\
			 B_{12}^t & A(G_2) & \cdots &  B_{2k}  \\
		\vdots  & \vdots  & \ddots   & \vdots\\
			B_{1k}^t	  & B_{2k}^t	 & \cdots \cdots  &A(G_k)
\end{bmatrix}
\end{equation}
where $B_{ij}$ are $0-1$ matrices of size $n_i\times n_j$.

For example, let $m =2, H = K_2, G_1 = P_3, G_2=P_4, I_1^{-1}(1)= \{v_1^1, v_1^2\}, I_1^{-1}(2)= \{v_1^3\}, I_2^{-1}(1) = \{v_2^1, v_2^2, v_2^3 \},$ and $ I_2^{-1}(2) = \{v_2^4\}$. Let $G = \bigvee_{H,\mathcal I}^{\mathcal{F}}$ be the $H_m$-join, then in the above mentioned ordering of the vertices of $G$,

$$ A(G)= 
  \left[\begin{array}{@{}ccc|cccc@{}}
    0 & 1 & 0 & 1 & 1 & 1 & 0 \\
    1 & 0 & 1 & 1 & 1 & 1 & 0 \\
    0 & 1 & 0 & 0 & 0 & 0 & 1 \\\hline
    1 & 1 & 0 & 0 & 1 & 0 & 0 \\
    1 & 1 & 0 & 1 & 0 & 1 & 0 \\
    1 & 1 & 0 & 0 & 1 & 0 & 1 \\
    0 & 0 & 1 & 0 & 0 & 1 & 0\\
  \end{array}\right].
 $$

\begin{rem}
We observe that neither of the non-diagonal blocks in the above matrix $A(G)$ is an all-one matrix. Hence, generally, the adjacency matrix of a graph $G$ obtained as an $H_m$-join is not necessarily in the form given in Equation \eqref{AofG}. Hence, in general, Theorem \ref{premain} cannot be used to calculate the spectrum of the graphs obtained from the $H_m$-join operation. However, we will show that the adjacency matrix of a graph $G$ obtained as an $H_m$-join has a similar form as in Equation \eqref{AofG}, which we discuss below.  
\end{rem}

\begin{lem} \label{Adjlem}
The adjacency matrix $A(G)$ given in \eqref{gen form} has the following form.
\begin{equation}\label{AoomfG}
A(G)=
\begin{bmatrix}
    A(G_1) &  \rho_{1,2}E_1E_2^t & \cdots  &  \rho_{1,k}E_1E_k^t \\
			\rho_{2,1} E_2 E_1^t & A(G_2) & \cdots &  \rho_{2,k}E_2E_k^t  \\
			\vdots  & \vdots  & \ddots   & \vdots\\
			\rho_{k,1} E_kE_1^t & \rho_{k,2} E_kE_2^t & \cdots \cdots  &A(G_k)
\end{bmatrix}
\end{equation}
    where for $1 \le j \le k$, $E_i$ is the $n_i \times m$ matrix  defined by $
(E_i)_{st}=\begin{cases}
    1 & \text{if}\ I_i(v_i^s)=t \\
    0 & \text{otherwise.}
\end{cases}$ and $\rho_{i,j}=\begin{cases}
    1 & \text{if} \ v_iv_j\in E(H)\\
    0 & \text{otherwise}
\end{cases}$.

\end{lem}
    
 \begin{pf} 

Note that, \begin{align*}
(E_iE_j^t)_{rs} &= \sum\limits_{q=1}^m (E_i)_{rq} (E_j^t)_{qs}\\
&=\begin{cases}
    1& \text{if}\ I_i(v_i^r)=I_j(v_j^s)\\
    0 & \text{otherwise.}  
\end{cases}. \\
&=(B_{ij})_{rs}.
\end{align*}
This completes the proof.
 \end{pf}
Given an $H_m$-join of graphs $\{G_1, G_2, \dots, G_k\}$ with indexing maps $I_1, I_2, \dots I_k$, the matrix $E_i$ shall be called the \it{indexing matrix} of $G_i$ for the map $I_i$. 
  \subsection{$E$-Main eigenvalues of a graph}

Recall that an eigenvalue $\lambda$ of the matrix $X(G)$ associated with the graph $G$ is called a \textit{main} eigenvalue if the eigenspace $\xi_{X(G)}(\lambda)$ is not orthogonal to the all-one vector $\bold 1_n$. Otherwise, it is called non-main.
  In the following definition, we define a generalization of main eigenvalues called the $E$-main eigenvalues for a rectangular matrix $E$. This notion plays a crucial role in our paper.

\begin{defn} \label{def E}
    Let $M$ be an $n \times n$ matrix over $\mathbb{C}$. Let $E$ be an $ n \times m$ matrix over $\mathbb{C}$. An eigenvalue $\lambda$ of $M$ is called an $E$-main eigenvalue of $M$ if the corresponding eigenspace $\xi_M(\lambda)$ is not orthogonal to the column space of $E$. Otherwise, it is called an $E$-non-main eigenvalue of $M$.
\end{defn}

 Suppose $G$ is the $H_m$-join of a family of graphs $\mathcal F = \{G_1,G_2,\dots,G_k\}$ with $ V(G_i) = \{v_i^{1},\dots,v_i^{n_i}\}$ and indexing maps $\mathcal I = \{I_1,\dots,I_k\}$.  
     For a subset $U$ of $V(G_i)$ and $1 \le s \le n_i$, the s-th coordinate of the  characteristic vector $\bold 1_U$ (of order $n_i \times 1$) is defined to be $$\bold (\bold 1_U)_s=\begin{cases}
    1 & \text{if}\ v_i^s \in U \\
    0 & \text{otherwise} \end{cases}.$$
    
   Define $U_{ij} := I_i^{-1}(j)$ which is a subset of $V(G_i)$, then 
       $V(G_i) = \sqcup_{j=1}^{m} U_{ij}$.
   Note that the indexing matrix $E_i$ is the $n_i \times m$ matrix with the characteristic vector $\bold 1_{U_{ij}}$ as its $j$-th column. We are particularly interested in the $E_i$-main eigenvalues of the adjacency matrix $A(G_i)$ in the context of $H_m$-joins owing to Lemma \ref{Adjlem}.

\begin{rem}\label{rem3}
     \begin{enumerate}  
    \item It is observed that the definition of $E_i$-main and $E_i$-non-main eigenvalues of the graph $G_i$ relies only on the graph $G_i$ and the indexing function $I_i$. In other words, this notion is unaffected by the maps in $\mathcal I \backslash \{I_i\}$ and the graphs in $\mathcal F \backslash \{G_i\}$.
   
    \item If $m=1$, then $U_{i1} = V(G_i)$ and $\bold 1_{U_{i1}} = \bold 1_{n_{i}}$.  In this case, an eigenvalue $\lambda$ of $A(G_i)$ is an $E_i$-non-main eigenvalue of $A(G_i)$ iff
     $\xi_{A(G_i)}(\lambda) \perp \bold{1}_{n_i}$ iff $\lambda$ is a non-main eigenvalue of $A(G_i)$. 
\item If every standard basis vector of $\mathbb{R}^{n_i}$ is a column of $E_i$, then $A(G_i)$ has no $E_i$-non-main eigenvalue.
\item Given a graph $G_i$ and an associated indexing matrix $E_i$, by \cite[Theorem 0.2]{sachs},  there is always an eigenvalue of $A(G_i)$ which is $E_i$-main. 
    \end{enumerate}
\end{rem}

\begin{defn}
    Let $M$ be a $n\times n$ matrix over $\mathbb{C}$. Let $U,V$ be $n\times m$ matrices over $\mathbb{C}$. Then the main function associated with $M,U,V$, denoted by $\Gamma_M(U,V) := V^t(\lambda I_n-M)^{-1}U \in M_{m}(\mathbb C(\lambda))$. When $U=V$, we write $\Gamma_M(U,V)=\Gamma_M(U)$.
\end{defn}

\section{Main Result}\label{mainresult}
 In this section, we prove the main theorem of this work. 
 
 \subsection{Proof of the main theorem} We first prove some basic lemmas.
 \begin{lem}\label{lem2}\cite{sachs}
Let $M$ be a complex matrix with block decomposition $M=
		\begin{bmatrix}
			A & B\\
			C & D
		\end{bmatrix}
		$. Then
  \begin{enumerate}[label=(\alph*)]
  \item $\det(M) = \det(A) \det(D-CA^{-1}B)$, if $A$ is invertible, and
   \item $\det(M) = \det(D) \det(A - BD^{-1}C)$, if $D$ is invertible.
  \end{enumerate}
	\end{lem}

The following lemma is a generalization of {\cite[Lemma 6]{msg}}.
	
	\begin{lem}\label{lem3}
		Let $A$ be an $n \times n$ invertible matrix, and $U$ and $V$ be $n \times m$ matrices. Then
		
		\begin{enumerate}[label=(\alph*)]
			\item $\det(I_n+UV^t) = \det(I_m+V^tU).$ 
			\item $\det(A + UV^{t}) =  \det(A)\det(I_m + V^{t}A^{-1}U).$ 
		\end{enumerate}
	\end{lem}	
	\begin{pf} 
 
\begin{eqnarray}
\begin{aligned}
\det(I_n+UV^t) 
& = \det\begin{bmatrix}
			I_m & -V^t \\ U & I_n 
		\end{bmatrix},\text{by Lemma} \ \ref{lem2}\text{(a)},\\
& =  \det(I_{m}+V^tU),\text{by Lemma} \ \ref{lem2}\text{(b)}.
\end{aligned}
\end{eqnarray}
\begin{eqnarray}
\begin{aligned}
\det(A + UV^{t}) &= \det(A(I_n+A^{-1}UV^t)),\\
&= \det(A)\det(I_n+(A^{-1}U)V^t),\\
&= \det(A)\det(I_m+V^tA^{-1}U), \text{by part (a)}. 
\end{aligned}
\end{eqnarray} \end{pf}

 The following is our main theorem, which expresses the spectrum of the adjacency matrix of a graph $G$ obtained as an $H_m$ join of graphs.

\begin{thm}\label{mainn}
 Consider a graph $H$ with vertex set $\{v_1,\dots,v_k\}$ and a family of graphs $\mathcal{F}=\{G_1,G_2, \hdots, G_k\}$ with $V(G_i)=\{v_i^1,v_i^2,\hdots,v_i^{n_i} \}$ along with the indexing maps $\mathcal I = \{I_1,I_2,\dots,I_k \}$. Let $\{E_1, E_2,\dots, E_k \}$ be the associated indexing matrices. Let $G$ be the $H_m$-join of the family of graphs $\mathcal F$ with respect to $\mathcal I$. For $1\leq i,j \leq k,$ let $\rho_{i,j}=\begin{cases}
     1 & \text{if} \ v_i v_j \in E(H),\\
     0 & \text{otherwise.}
 \end{cases}$
 For $1\leq i \leq k$, let $\phi_{A(G_i)}(\lambda)=\phi_i := \det(\lambda I_{n_i}-A(G_i))$ be the characteristic polynomial of matrix $A(G_i)$.  Let $n := \sum_{i=1}^k n_i$ and $\Gamma_i: = \Gamma_{A(G_i)}(E_i)$. Then 
 \begin{equation}\label{maine}
			\det(\lambda I_n - A(G)) = \Bigg( \prod_{i=1}^k \phi_i \Bigg) \det(\widetilde{A}) .
		\end{equation}  where
 $$
		\widetilde{A} :=  \begin{bmatrix}
			I_m & -\rho_{1,2}{\Gamma_1}  & \cdots & -\rho_{1,k}{\Gamma_1}  \\
			-\rho_{2,1}{\Gamma_2}  & I_m & \cdots & -\rho_{2,k}{\Gamma_2} \\
			\vdots & \vdots & \ddots & \vdots  \\
			-\rho_{k,1}{\Gamma_k} & -\rho_{k,2}{\Gamma_k} & \cdots &I_m
		\end{bmatrix}.$$

\end{thm}

\begin{pf}
    From Equation (\ref{AoomfG}), we have
\begin{equation*}
A(G)=
\begin{bmatrix}
    A(G_1) &  \rho_{1,2}E_1E_2^t & \cdots  &  \rho_{1,k}E_1E_k^t \\
			\rho_{2,1} E_2 E_1^t & A(G_2) & \cdots &  \rho_{2,k}E_2E_k^t  \\
			\vdots  & \vdots  & \ddots   & \vdots\\
			\rho_{k,1} E_kE_1^t & \rho_{k,2} E_kE_2^t & \cdots \cdots  &A(G_k)
\end{bmatrix}
    \end{equation*} We shall prove the claim by induction on $k$.
For $k=2$, using Lemma \ref{lem2}(b),  we obtain
		\begin{align*}
			\begin{vmatrix}
				\lambda I_{n_1}-A(G_1) & -\rho_{1,2}E_1E_2^t \\
				-\rho_{2,1}E_2E_1^t & \lambda I_{n_2}- A(G_2)
			\end{vmatrix} &=  \phi_2 \det \big((\lambda I_{n_1} - A(G_1)) - \rho_{1,2}\rho_{2,1}E_1 E_2^t (\lambda I_{n_2}-A(G_2))^{-1}E_2 E_1^t\big)\\
			&= \phi_2\det \big((\lambda I_{n_1} - A(G_1) - \rho_{1,2}\rho_{2,1}(E_1 \Gamma_2) E_1^t)\big)\\
			&=  \phi_1 \phi_2 \det(I_{m} -  \rho_{1,2}\rho_{2,1}E_1^{t} (\lambda I_{n_1} - A(G_1))^{-1}E_1\Gamma_2), \text{ by Lemma } \ref{lem3} \text{ (b)}\\
			&= \phi_1 \phi_2 \det(I_{m}-\rho_{1,2}\rho_{2,1}\Gamma_1\Gamma_2)\\
			&= \phi_1 \phi_2 \begin{vmatrix}
				I_{m} & -\rho_{1,2}\Gamma_1 \\
				-\rho_{2,1}\Gamma_2 & I_m
			\end{vmatrix}, \text{ by Lemma } \ref{lem2}(b).
		\end{align*}

This completes the proof of the claim for $k=2.$

Now, for $k \ge 3 $, again by Lemma \ref{lem2}(b) we obtain 
		\begin{equation} \label{pfind}
			\det(\lambda I_n -A(G))=\det(\lambda I_{n_k} - A(G_k))\det(S)
		\end{equation}
		\text{where }
		\begin{eqnarray}
			\begin{aligned}
				S &=\begin{bmatrix}
					\lambda I_{n_1} -	A(G_1) & - \rho_{1,2}E_1E_2^t & \cdots & - \rho_{1,k-1}E_1E_{k-1}^t \\
					-\rho_{2,1}E_2E_1^t & \lambda I_{n_2}-A(G_1) & \cdots & -\rho_{2,k-1}E_2E_{k-1}^t \\
					\vdots & \vdots & \ddots & \vdots  \\
					- \rho_{k-1,1}E_{k-1}E_1^t & - \rho_{k-1,2}E_{k-1}E_2^t & \cdots & \lambda I_{n_{k-1}}-A(G_{k-1}) &
				\end{bmatrix}\\ 
				&-\begin{bmatrix}
					-\rho_{1,k}E_1E_k^t \\
					-\rho_{2,k}E_2E_k^t \\
					\vdots  \\
					-\rho_{k-1,k}E_{k-1}E_k^t
				\end{bmatrix} (\lambda I_{n_k} - A(G_k))^{-1}
				\begin{bmatrix}
					-\rho_{k,1}E_kE_1^t & -\rho_{k,2}E_kE_2^t & \cdots & -\rho_{k,k-1}E_{k}E_{k-1}^t
				\end{bmatrix}
			\end{aligned}
		\end{eqnarray}

 Therefore,
  
		$\det(S) = \det\Big(\begin{bmatrix}
			\lambda I_{n_1} -	A(G_1) & -\rho_{1,2}E_1E_2^t & \cdots & -\rho_{1,k-1}E_1E_{k-1}^t \\
			-\rho_{2,1}E_2E_1^t & \lambda I_{n_2}-A(G_2) & \cdots & -\rho_{2,k-1}E_2E_{k-1}^t \\
			\vdots & \vdots & \ddots & \vdots  \\
			-\rho_{k-1,1}E_{k-1}E_1^t & -\rho_{k-1,2}E_{k-1}E_2^t & \cdots & \lambda I_{n_{k-1}}-A(G_{k-1}) &
		\end{bmatrix}\\ \newline -  \begin{bmatrix} 
			\rho_{1,k}\rho_{k,1}E_1\Gamma_k E_1^t &  \rho_{1,k}\rho_{k,2} E_1\Gamma_k E_2^t & \hdots & \rho_{1,k}\rho_{k,k-1}E_1\Gamma_k E_{k-1}^t \\
			\rho_{2,k}\rho_{k,1}E_2\Gamma_k E_1^t &  \rho_{2,k}\rho_{k,2}E_2\Gamma_k E_2^t & \hdots & \rho_{2,k}\rho_{k,k-1}E_2\Gamma_k E_{k-1}^t \\
			\vdots & \vdots & \ddots & \vdots  \\
			\rho_{k-1,k}\rho_{k,1}E_{k-1}\Gamma_k E_1^t &  \rho_{k-1,k}\rho_{k,2}E_{k-1}\Gamma_k E_2^t & \hdots & \rho_{k-1,k}\rho_{k,k-1}E_{k-1}\Gamma_k E_{k-1}^t \end{bmatrix}	\Big) \\
$
		
		$= \det\Big( \begin{bmatrix}
			\lambda I_{n_1} -	A(G_1) & \bold 0 & \dots & \bold 0 \\ 
			\bold 0 &  \lambda I_{n_2}-A(G_2) &  & \bold 0\\ 
			\vdots& &  \ddots & \vdots \\ 
			\bold 0 & \bold 0 & \cdots  & \lambda I_{n_{k-1}}-A(G_{k-1}) 
		\end{bmatrix} \newline -
		\begin{bmatrix} 
			\rho_{1,k}\rho_{k,1}E_1\Gamma_k E_1^t & E_1( \rho_{1,2} + \rho_{1,k}\rho_{k,2}\Gamma_k) E_2^t & \hdots &  E_1( \rho_{1,k-1} +\rho_{1,k}\rho_{k,k-1}  \Gamma_k) E_{k-1}^t \\
			E_2( \rho_{2,1} + \rho_{2,k}\rho_{k,1} \Gamma_k) E_1^t & \rho_{2,k}\rho_{k,2}E_2\Gamma_k E_2^t & \hdots & E_2( \rho_{2,k-1}+\rho_{2,k}\rho_{k,k-1}\Gamma_k) E_{k-1}^t \\
			\vdots & \vdots & \ddots & \vdots  \\
			E_{k-1}( \rho_{k-1,1}+\rho_{k-1,k}\rho_{k,1} \Gamma_k) E_1^t & 	E_{k-1}( \rho_{k-1,2}+\rho_{k-1,k}\rho_{k,2}\Gamma_k) E_2^t & \hdots & \rho_{k-1,k}\rho_{k,k-1}E_{k-1}\Gamma_k E_{k-1}^t \end{bmatrix}	\Big)	$
		
		$= \det\Big( \begin{bmatrix}
			\lambda I_{n_1} -	A(G_1) &  &  & \\ 
			&  \lambda I_{n_2}-A(G_2) &  & \\ 
			&  &  \ddots & \\ 
			&  &   & \lambda I_{n_{k-1}}-A(G_{k-1}) 
		\end{bmatrix} \newline -
		\begin{bmatrix} 
			\rho_{1,k}\rho_{k,1}E_1\Gamma_k & E_1( \rho_{1,2} + \rho_{1,k}\rho_{k,2}\Gamma_k) & \hdots &  E_1( \rho_{1,k-1} +\rho_{1,k}\rho_{k,k-1}  \Gamma_k) \\
			E_2( \rho_{2,1} + \rho_{2,k}\rho_{k,1} \Gamma_k) & \rho_{2,k}\rho_{k,2}E_2\Gamma_k & \hdots & E_2( \rho_{2,k-1}+\rho_{2,k}\rho_{k,k-1}\Gamma_k)  \\
			\vdots & \vdots & \ddots & \vdots  \\
			E_{k-1}( \rho_{k-1,1}+\rho_{k-1,k}\rho_{k,1} \Gamma_k) & 	E_{k-1}( \rho_{k-1,2}+\rho_{k-1,k}\rho_{k,2}\Gamma_k) & \hdots & \rho_{k-1,k}\rho_{k,k-1}E_{k-1}\Gamma_k  \end{bmatrix}	\Big) \newline \times	\begin{bmatrix}
			E_1^t &  &  & \\ 
			&  E_2^t &  & \\ 
			&  &  \ddots & \\ 
			&  &   & E_{k-1}^t
			
		\end{bmatrix}\Big)$
		
		Using Lemma \ref{lem3}(b),\\
		$= {\displaystyle{\prod_{i=1}^{k-1}} \phi_i} \cdot \det\Big( I_{(k-1)m} - \begin{bmatrix}
			E_1^t &  &  & \\ 
			&  E_2^t &  & \\ 
			&  &  \ddots & \\ 
			&  &   & E_{k-1}^t
		\end{bmatrix} \newline \times
         \begin{bmatrix}  (\lambda I_{n_1} -	A(G_1))^{-1} &  &  & \\ 
			&  (\lambda I_{n_2}-A(G_2))^{-1} &  & \\ 
			&  &  \ddots & \\ 
			&  &   & (\lambda I_{n_{k-1}}-A(G_{k-1}))^{-1} 
			
		\end{bmatrix} \newline \times
		\begin{bmatrix} 
			\rho_{1,k}\rho_{k,1}E_1\Gamma_k & E_1( \rho_{1,2} + \rho_{1,k}\rho_{k,2}\Gamma_k) & \hdots &  E_1( \rho_{1,k-1} +\rho_{1,k}\rho_{k,k-1}  \Gamma_k) \\
			E_2( \rho_{2,1} + \rho_{2,k}\rho_{k,1} \Gamma_k) & \rho_{2,k}\rho_{k,2}E_2\Gamma_k & \hdots & E_2( \rho_{2,k-1}+\rho_{2,k}\rho_{k,k-1}\Gamma_k)  \\
			\vdots & \vdots & \ddots & \vdots  \\
			E_{k-1}( \rho_{k-1,1}+\rho_{k-1,k}\rho_{k,1} \Gamma_k) & 	E_{k-1}( \rho_{k-1,2}+\rho_{k-1,k}\rho_{k,2}\Gamma_k) & \hdots & \rho_{k-1,k}\rho_{k,k-1}E_{k-1}\Gamma_k  \end{bmatrix}	\Big)
		$
		
		$= {\displaystyle{\prod_{i=1}^{k-1}} \phi_i} \cdot \det\Big( I_{(k-1)m} - \newline 	\begin{bmatrix} 
			\rho_{1,k}\rho_{k,1}\Gamma_1\Gamma_k & \Gamma_1( \rho_{1,2} + \rho_{1,k}\rho_{k,2}\Gamma_k) & \hdots &  \Gamma_1( \rho_{1,k-1} +\rho_{1,k}\rho_{k,k-1}  \Gamma_k) \\
			\Gamma_2( \rho_{2,1} + \rho_{2,k}\rho_{k,1} \Gamma_k) & \rho_{2,k}\rho_{k,2}E_2\Gamma_k & \hdots & \Gamma_2( \rho_{2,k-1}+\rho_{2,k}\rho_{k,k-1}\Gamma_k)  \\
			\vdots & \vdots & \ddots & \vdots  \\
			\Gamma_{k-1}( \rho_{k-1,1}+\rho_{k-1,k}\rho_{k,1} \Gamma_k) & 	\Gamma_{k-1}( \rho_{k-1,2}+\rho_{k-1,k}\rho_{k,2}\Gamma_k) & \hdots & \rho_{k-1,k}\rho_{k,k-1}\Gamma_{k-1}\Gamma_k  \end{bmatrix}	\Big)
		$

		$= {\displaystyle{\prod_{i=1}^{k-1}} \phi_i} . \det\Big( \begin{bmatrix}
			I_m & -\rho_{1,2}\Gamma_1 & \hdots & -\rho_{1,k-1}\Gamma_1 \\
			-\rho_{2,1}\Gamma_2 & I_m & \hdots & -\rho_{2,k-1}\Gamma_2 \\
			\vdots & \vdots & \ddots & \vdots \\
			-\rho_{k-1,1}\Gamma_{k-1} & -\rho_{k-1,2}\Gamma_{k-1} & \hdots & I_m
		\end{bmatrix} - \newline \begin{bmatrix}
			\rho_{1,k}\rho_{k,1}\Gamma_1\Gamma_k & \rho_{1,k}\rho_{k,2}\Gamma_1\Gamma_k & \hdots & \rho_{1,k}\rho_{k,k-1}\Gamma_1\Gamma_k \\
			\rho_{2,k}\rho_{k,1}\Gamma_2\Gamma_k & \rho_{2,k}\rho_{k,2}\Gamma_2\Gamma_k & \hdots & \rho_{2,k}\rho_{k,k-1}\Gamma_2\Gamma_k \\
			\vdots & \vdots & \ddots & \vdots \\
			\rho_{k-1,k}\rho_{k,1}\Gamma_{k-1}\Gamma_k & \rho_{k-1,k}\rho_{k,2}\Gamma_{k-1}\Gamma_k & \hdots &\rho_{k-1,k}\rho_{k,k-1}\Gamma_{k-1}\Gamma_k
		\end{bmatrix} \Big)$
		
		$= {\displaystyle{\prod_{i=1}^{k-1}} \phi_i} . \det\Big( \begin{bmatrix}
			I_m & -\rho_{1,2}\Gamma_1 & \hdots & -\rho_{1,k-1}\Gamma_1 \\
			-\rho_{2,1}\Gamma_2 & I_m & \hdots & -\rho_{2,k-1}\Gamma_2 \\
			\vdots & \vdots & \ddots & \vdots \\
			-\rho_{k-1,1}\Gamma_{k-1} & -\rho_{k-1,2}\Gamma_{k-1} & \hdots & I_m
		\end{bmatrix}  - \newline \begin{bmatrix}
			\rho_{1,k}\Gamma_1 \\ \rho_{2,k}\Gamma_2 \\ \vdots \\\rho_{k-1,k} \Gamma_{k-1} 
		\end{bmatrix}
\begin{matrix}
    {\Gamma_k}
\end{matrix}
 \begin{bmatrix}
			\rho_{k,1} & \rho_{k,2} & \hdots & \rho_{k,k-1}
		\end{bmatrix} \Big)$
		
		$= {\displaystyle{\prod_{i=1}^{k-1}} \phi_i} . \det\Big(
		\begin{bmatrix}
			I_m & -\rho_{1,2}{\Gamma_1}  & \cdots & -\rho_{1,k}{\Gamma_1}  \\
			-\rho_{2,1}{\Gamma_2}  & I_m & \cdots & -\rho_{2,k}{\Gamma_2} \\
			\vdots & \vdots & \ddots & \vdots  \\
			-\rho_{k,1}{\Gamma_k} & -\rho_{k,2}{\Gamma_k} & \cdots & I_m
		\end{bmatrix}\Big),~\text{by Lemma } \ref{lem2}
		$.
  
Now, from Equation(\ref{pfind}),
		     $$\det(\lambda I_n-A(G))=\prod_{i=1}^{k} \phi_i . \det\Big(
		\begin{bmatrix}
			I_m & -\rho_{1,2}{\Gamma_1}  & \cdots & -\rho_{1,k}{\Gamma_1}  \\
			-\rho_{2,1}{\Gamma_2}  & I_m & \cdots & -\rho_{2,k}{\Gamma_2} \\
			\vdots & \vdots & \ddots & \vdots  \\
			-\rho_{k,1}{\Gamma_k} & -\rho_{k,2}{\Gamma_k} & \cdots & I_m
		\end{bmatrix}\Big) $$
which proves the theorem.
\end{pf}  
\begin{rem}

For $i \in [k],$ let $\{\theta_i^1, \theta_i^2, \cdots, \theta_i^{k_i}\}$ and $\{\theta_i^1, \theta_i^2, \cdots, \theta_i^{m_i}\}$  ($m_i \le k_i$) be respectively the distinct eigenvalues and the distinct $E_i$-main eigenvalues of the graph $G_i$.  Let the spectral decomposition of $A(G_i)=\Sigma_{j=1}^{k_i} \theta_i^j \cdot \pi_{\theta_i^j}$, where $\pi_{\theta_i^j}$ is the orthogonal projection onto 
the eigenspace $\xi_{A(G_i)}(\theta_i^j)$ corresponding to the eigenvalue $\theta_i^j$. Then the spectral decomposition of $(\lambda I - A(G_i))^{-1}=\Sigma_{j=1}^{k_i} (\frac{1}{\lambda-\theta_i^j}\cdot\pi_{\theta_i^j})$. Therefore, $\Gamma_i = E_i^t(\lambda I - A(G_i))^{-1}E_i=\Sigma_{j=1}^{k_i} (\frac{1}{\lambda-\theta_i^j}\cdot E_i^t \pi_{\theta_i^j}E_i).$ 

We will prove that 
$E_i^t \pi_{\theta_i^j} E_i \ne 0$ if, and only if, $\pi_{\theta_i^j} E_i \ne 0$. The proof of the only if part is clear, and we will prove the if part.  Write $\mathbb{R}^{n_i}=W \oplus W^{\perp}$ where $W=\xi_{A(G_i)}(\theta_i^j)$. Let $u=u_W + u_{W^{\perp}}$ be the unique representation of a vector $u \in \mathbb{R}^{n_i}$ in this direct sum. Suppose $\pi_{\theta_i^j} E_i \ne 0$. Then there exists $s$ such that $\pi_{\theta_i^j} (u_i^s) = (u_i^s)_W \ne 0$ where $u_i^s$ is the $s$-th column of $E_i$. 
Thus
\begin{align*}
(u_i^{s})^t \pi_{\theta_i^j} u_i^s=<\pi_{\theta_i^j} u_i^s,u_i^s>&=<\pi_{\theta_i^j} u_i^s, (u_i^s)_W + (u_i^s)_{W^{\perp}}>\\
&=<(u_i^s)_W,(u_i^s)_W + (u_i^s)_{W^{\perp}}>\\ &=< (u_i^s)_W, (u_i^s)_W> \\&\neq 0.\end{align*}
That is, $(u_i^{s})^t \pi_{\theta_i^j} u_i^s=(E_i^t \pi_{\theta_i^j} E_i)_{ss}$ is non-zero and our claim follows. 
Thus
$E_i^t \pi_{\theta_i^j} E_i \ne 0$ if, and only if, $\pi_{\theta_i^j} E_i \ne 0$,
if, and only if, $\theta_i^j$ is an $E_i$-main eigenvalue of $G_i$. Therefore, $\Gamma_i=\Sigma_{j=1}^{m_i} (\frac{1}{\lambda-\theta_i^j}\cdot E_i^t\pi_{\theta_i^j}E_i)$. 

 \begin{equation}\label{defnfg}
\text{This implies that, for } i\in [k], \  \Gamma_i=\frac{1}{g_i} \cdot f_i  \text{ where } g_i=\prod_{j=1}^{m_i}(\lambda - \theta_i^j),\ f_i \in M_{m}(\mathbb C[\lambda]).
\end{equation}

We have $\det(\frac{1}{g_i}\cdot f_i) = \frac{1}{g_i^m} \det(f_i)$ and hence,
\begin{equation}\label{4.7}
 \det(\lambda I_n-A(G))      =\prod_{i=1}^{k} \frac{\phi_i}{g_i^m}.\begin{vmatrix}
g_1 I_m & -\rho_{1,2}f_1  & \cdots & -\rho_{1,k}f_1\\
-\rho_{2,1}f_2  & g_2 I_m & \cdots & -\rho_{2,k}f_2 \\
\vdots & \vdots & \ddots & \vdots  \\
-\rho_{k,1}f_k & -\rho_{k,2}f_k& \cdots &g_k I_m
\end{vmatrix},\end{equation}
where
\begin{equation}\label{4.8}
    \Phi(\lambda) := \begin{vmatrix}
g_1 I_m & -\rho_{1,2}f_1  & \cdots & -\rho_{1,k}f_1\\
-\rho_{2,1}f_2  & g_2 I_m & \cdots & -\rho_{2,k}f_2 \\
\vdots & \vdots & \ddots & \vdots  \\
-\rho_{k,1}f_k & -\rho_{k,2}f_k& \cdots &g_k I_m
\end{vmatrix} \in \mathbb{C}[\lambda].\end{equation} 
    
Hence, we have the following from Equation (\ref{4.7}). 
\end{rem}
\begin{thm}\label{mainnn}
    Let the notations be as above. We observe the following about the spectrum of the matrix $A(G)$. Let $\lambda$ be an eigenvalue of the graph $G_i$ with multiplicity $\mult_i(\lambda)$.
    \begin{itemize}
\item If $\lambda$ is $E_i$-non-main eigenvalue of $A(G_i)$, then $\lambda$ is an eigenvalue of $A(G)$ with multiplicity at least $\mult_i(\lambda)$.
\item If $\lambda$ is an $E_i$-main eigenvalue of $A(G_i)$, then $\lambda$ is an eigenvalue of $A(G)$ with multiplicity at least $\mult_i(\lambda)-m$.
\item Remaining eigenvalues are the roots of the polynomial $\Phi(\lambda) \in \mathbb{C}[\lambda] $ (defined above).
\end{itemize}
\end{thm}

Next, we illustrate our result with a few examples.

\begin{example}
     Consider the following $(P_2)_2$-join $G$ of $G_1=K_2, G_2=K_5$ with the indexing maps $I_1,I_2$. Below, each vertex $v$ is labeled with the values of these indexing maps.
\begin{figure}[h]
\centering
\begin{tikzpicture}[
every edge/.style = {draw=black,very thick},
 vrtx/.style args = {#1/#2}{%
      circle, draw, scale=.5, fill=black,
      minimum size, label=#1:#2}
                    ]
\node(a) [vrtx=left/1] at (-2,-1) {};
\node(b) [vrtx=left/1] at (-2,3) {};
\node(2a) [vrtx=above/1] at (4,3) {};
\node(2b) [vrtx=below right/2] at (6,1) {};
\node(2c) [vrtx=below/2] at (5,-1) {};
\node(2d) [vrtx=below/1] at (3,-1) {};
\node(2e) [vrtx=below left/1] at (2,1) {};

\path (a)edge(b)
(2a)edge(2b)edge(2c)edge(2d)edge(2e)
(2b)edge(2c)edge(2d)edge(2e)
(2c)edge(2d)edge(2e)
(2d)edge(2e)
(a)edge[blue](2a)edge[blue](2d)edge[blue](2e)
(b)edge[blue](2a)edge[blue](2d)edge[blue](2e);
 \end{tikzpicture}\caption{$(P_2)_2$-join of $\{K_2, K_5\}$}
                    \end{figure}
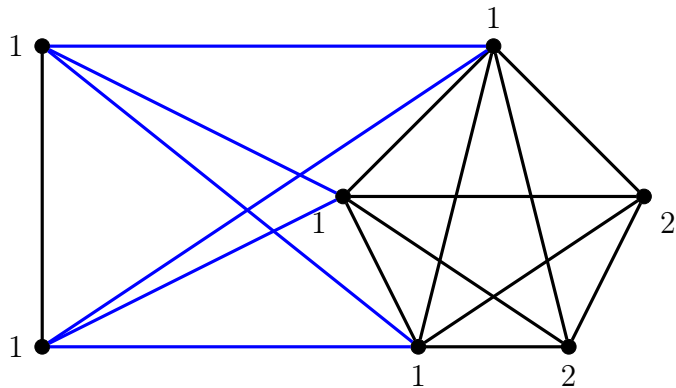           

The indexing matrices, with a suitable ordering, are as follows.
$$E_1=\begin{bmatrix}
    1 & 0 \\
    1 & 0 
\end{bmatrix}, E_2=\begin{bmatrix}
      1 & 0 \\  
      1 & 0 \\
      1 & 0 \\
      0 & 1 \\
      0 & 1
\end{bmatrix}.$$

We find that $$\phi_1(\lambda)=\lambda^2 -1, \phi_2 (\lambda)=(\lambda-4) (\lambda +1)^4.$$

and

$$\Gamma_1=\frac{1}{\lambda^2-1}\begin{bmatrix}
    2\lambda+2 & 0\\
    0 & 0
\end{bmatrix}, \Gamma_2=\frac{1}{(\lambda-4)(\lambda+1)}\begin{bmatrix}
    3(x-1) & 6\\
    6 & 2(x-2)
\end{bmatrix}.$$

Now, by Theorem \ref{mainn}, 
\begin{align*}
    \det(\lambda I_7-A(G))&=\phi_1(\lambda) \cdot \phi_2(\lambda)  \cdot \det\begin{bmatrix}
    I_2 & -\Gamma_1 \\
    -\Gamma_2 & I_2
\end{bmatrix}, \\& = (\lambda^2 -1)\cdot (\lambda-4) (\lambda +1)^4 \cdot \frac{(\lambda+2)(\lambda-5)}{(\lambda+1)(\lambda-4)},\\
&=(\lambda+2)(\lambda-5)(\lambda-1)(\lambda+1)^4.
\end{align*}
\end{example}

\begin{rem} Note that
    \begin{itemize}[leftmargin=*]
        \item $\sigma(A(G_1))=\{1,-1\}$ with $\xi_{A(G_1)}(1)=\text{span}\{\begin{bmatrix}
       1\\1
   \end{bmatrix} \}$, $\xi_{A(G_1)}(-1)=\text{span}\{\begin{bmatrix}
       1\\-1
   \end{bmatrix} \}$.
   Thus $1\in \sigma(A(G_1))$ is $E_1$-main and $-1\in \sigma(A(G_1))$ is $E_1$-non-main.

   \item $\sigma(A(G_2))=\{4,-1,-1,-1,-1\}$ with  $\xi_{A(G_2)}(4)=\text{span}\{\begin{bmatrix}
       1\\1\\1\\1\\1
   \end{bmatrix} \}$,\\  $\xi_{A(G_2)}(-1)=\text{span}\{\begin{bmatrix}
       1\\-1\\0\\0\\0
   \end{bmatrix}, \begin{bmatrix}
       0\\1\\-1\\0\\0
   \end{bmatrix}, \begin{bmatrix}
      0\\0\\ 1\\-1\\0
   \end{bmatrix}, \begin{bmatrix}
      0\\0\\0\\ 1\\-1
   \end{bmatrix}    \}$. Thus the eigenvalues $4$ and $-1$ are both $E_2$-main.

   \item Note that -1 gets carried forward at least three times, by Theorem \ref{mainnn}. One time as an $E_1$-non-main eigenvalue of $A(G_1)$, at least two times as an $E_2$ main eigenvalue of $A(G_2)$. 
    \end{itemize}
\end{rem}

\begin{example}\label{mainexample}
    Consider the following $(P_3)_3$-join $G$ of $G_1=K_2, G_2=P_3,  G_3=K_{1,3}$ with the indexing maps $I_1,I_2,I_3,I_4$. Below, each vertex $v$ is labeled with the values of these indexing maps.
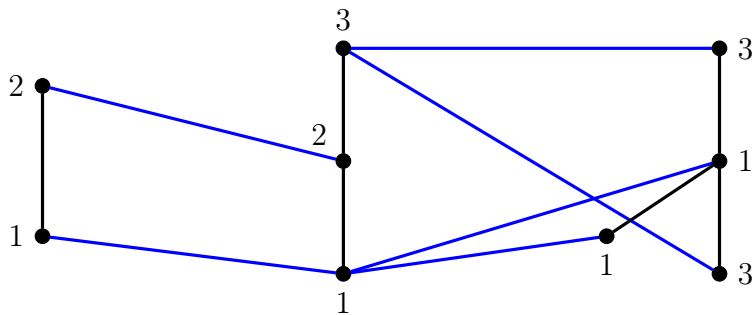
\begin{figure}[h]
\centering
\begin{tikzpicture}[
every edge/.style = {draw=black,very thick},
 vrtx/.style args = {#1/#2}{%
      circle, draw, scale=.5, fill=black,
      minimum size, label=#1:#2}
                    ]
\node(a) [vrtx=left/1] at (0,0) {};
\node(b) [vrtx=left/2] at (0,2) {};
\node(2a) [vrtx=below /1] at (4,-.5) {};
\node(2b) [vrtx=above left/2] at (4,1) {};
\node(2c) [vrtx=above/3] at (4,2.5) {};
\node(3a) [vrtx=right/1] at (9,1) {};
\node(3b) [vrtx=below/1] at (7.5,0) {};
\node(3c) [vrtx=right/3] at (9,2.5) {};
\node(3d) [vrtx=right/3] at (9,-.5) {};

\path (a)edge(b)edge[blue](2a)
(b)edge[blue](2b)
(2a)edge(2b)edge[blue](3a)edge[blue](3b)
(2b)edge(2c)
(2c)edge[blue](3d)edge[blue](3c)
(3a)edge(3b)edge(3c)edge(3d)
;

 \end{tikzpicture}\caption{$(P_3)_3$-join of $\{K_2, P_3, K_{1,3}\}$}
                    \end{figure}           
\end{example}
The indexing matrices, with a suitable ordering, are as follows.
$$E_1=\begin{bmatrix}
    1 & 0 & 0\\
    0 & 1 & 0
\end{bmatrix}, E_2=\begin{bmatrix}
       1 & 0 & 0\\
       0 & 1 & 0\\
    0 & 0 & 1
\end{bmatrix}, E_3=\begin{bmatrix}
      1 & 0 & 0\\  1 & 0 & 0\\
      0 & 0 & 1\\
      0 & 0 & 1
\end{bmatrix}.$$

We have  
$$ A(G_1)=\begin{bmatrix}
    0 & 1 \\
    1 & 0
\end{bmatrix}, A(G_2)= \begin{bmatrix}
    0 & 1 & 0\\
    1 & 0 & 1\\
    0 & 1 & 0
\end{bmatrix}, A(G_3)= \begin{bmatrix}
    0 & 1 & 1 & 1\\
    1 & 0 & 0 & 0\\
    1 & 0 & 0 & 0\\
    1 & 0 & 0 & 0
\end{bmatrix},$$ 

with respective characteristic polynomials

$$\phi_1(\lambda)=\lambda^2 -1, \phi_2(\lambda)=\lambda^3 -2\lambda, \phi_3 (\lambda)=\lambda^2 (\lambda^2 -3),$$
and 

$$A(G)=\begin{bmatrix}
    0 & 1 & 1 & 0 & 0 & 0 & 0 & 0 & 0 \\
    1 & 0 & 0 & 1 & 0 & 0 & 0 & 0 & 0 \\
    1 & 0 & 0 & 1 & 0 & 1 & 1 & 0 & 0 \\
    0 & 1 & 1 & 0 & 1 & 0 & 0 & 0 & 0 \\
    0 & 0 & 0 & 1 & 0 & 0 & 0 & 1 & 1 \\
    0 & 0 & 1 & 0 & 0 & 0 & 1 & 1 & 1 \\
    0 & 0 & 1 & 0 & 0 & 1 & 0 & 0 & 0 \\
    0 & 0 & 0 & 0 & 1 & 1 & 0 & 0 & 0 \\
    0 & 0 & 0 & 0 & 1 & 1 & 0 & 0 & 0
\end{bmatrix}.$$ 
Also, $$\Gamma_1=\frac{1}{\lambda^2-1}\begin{bmatrix}
    \lambda & 1 & 0\\
    1 & \lambda & 0\\
    0 & 0 & 0
\end{bmatrix},\ \Gamma_2=\frac{1}{\lambda^3-2\lambda}\begin{bmatrix}
    \lambda^2-1 & \lambda & 1\\
    \lambda & \lambda^2 & \lambda\\
    1 & \lambda & \lambda^2-1
\end{bmatrix},$$ and $$
\Gamma_3=\frac{1}{\lambda (\lambda^2 -3)}\begin{bmatrix}
   2\lambda^2+2\lambda-2 & 0 & 2\lambda+2\\
   0 & 0 & 0\\
   2\lambda+2 & 0 & 2\lambda^2-2
\end{bmatrix}.$$
Now, by Theorem \ref{mainn}, 
\begin{align*}
    \det(\lambda I_9-A(G))&=\phi_1(\lambda) \cdot \phi_2(\lambda) \cdot \phi_3(\lambda) \cdot \det\begin{bmatrix}
    I_3 & -\Gamma_1 & 0\\
    -\Gamma_2 & I_3 & -\Gamma_2\\
    0 & -\Gamma_3 & I_3
\end{bmatrix}, \\& = (\lambda^2 -1)\cdot (\lambda^3 -2\lambda) \cdot (\lambda^2 (\lambda^2 -3)) \cdot \frac{(\lambda^8 - 12 \lambda^6 - 2 \lambda^5 + 39 \lambda^4 + 6 \lambda^3 - 34 \lambda^2 - 10 \lambda + 2)}{(\lambda^2 -1)\cdot (\lambda^3 -2\lambda) \cdot (\lambda (\lambda^2 -3))},
\\  &= \lambda (\lambda^8 - 12 \lambda^6 - 2 \lambda^5 + 39 \lambda^4 + 6 \lambda^3 - 34 \lambda^2 - 10 \lambda + 2).
\end{align*}

\begin{rem}

Note that 

\begin{itemize}[leftmargin=*]
   \item $\sigma(A(G))=\{-2.2326..., -2.2095..., -0.9057..., -0.50631..., 0, 0.1381..., 1.3395..., 1.5942..., 2.8763... \}$, .
   \item $\sigma(A(G_1))=\{1,-1\}$ with $\xi_{A(G_1)}(1)=\text{span}\{\begin{bmatrix}
       1\\1
   \end{bmatrix} \}$, $\xi_{A(G_1)}(-1)=\text{span}\{\begin{bmatrix}
       1\\-1
   \end{bmatrix} \}$.
   \item $\sigma(A(G_2))=\{0,\sqrt 2, -\sqrt 2\}$ with $\xi_{A(G_2)}(0)=\text{span}\{ \begin{bmatrix}
       1\\0\\-1
   \end{bmatrix}\}$,  $\xi_{A(G_2)}(\sqrt 2)=\text{span}\{ \begin{bmatrix}
       1\\\sqrt 2\\1
   \end{bmatrix}\}$, $\xi_{A(G_2)}(-\sqrt 2)=\text{span}\{ \begin{bmatrix}
       -1\\\sqrt 2\\-1
   \end{bmatrix}\}$.
\item $\sigma(A(G_3))=\{0,0,+\sqrt 3,-\sqrt 3 \}$ with\newline
$\xi_{A(G_3)}(0)=\text{span}\{\begin{bmatrix}
    0 \\ -1 \\ 1 \\ 0
\end{bmatrix}, \begin{bmatrix}
    0 \\ -1 \\ 0 \\ 1
\end{bmatrix} \}, \xi_{A(G_3)}(\sqrt 3)=\text{span}\{\begin{bmatrix}
    \sqrt{3} \\ 1 \\ 1 \\ 1
\end{bmatrix}   \},  \xi_{A(G_3)}(-\sqrt 3)=\text{span}\{\begin{bmatrix}
    -\sqrt{3} \\ 1 \\ 1 \\ 1
\end{bmatrix}   \}. $

 \item  By Remark \ref{rem3}(3), All eigenvalues of $A(G_1)$, $A(G_2)$ are $E_1$-main and $E_2$-main, respectively.
We also see that all eigenvalues of $A(G_3)$ are $E_3$-main.
\item We observe that 0 is an $E_2$-main and $E_3$-main eigenvalue of $A(G_2)$ and $A(G_3)$, respectively. By our result, $0$ can get carried forward as an eigenvalue of $A(G_2)$ at least $-2 = 1-3$ times and as an eigenvalue of $A(G_3)$ at least $-1 = 2-3$ times. Here, we emphasize ``at least" because $0$ is an eigenvalue of $A(G)$.
\end{itemize} 
\end{rem}
\subsection{Reducing the value of $m$}

In the above theorem, we observe that the lesser the value of $m$ is, the more the number of eigenvalues from factor graphs that get carried forward. Below, we discuss a few situations when the value of $m$ can be reduced.
Consider an $H_m$-join $G$ of the family $\mathcal F = \{G_1,G_2,\dots,G_k\}$ with respect to the indexing maps $\mathcal I = \{I_1,\dots,I_k\}$. 

\textbf{Situation 1:} Suppose $t \in [m]$ is not in $\cup_{i=1}^k I_i(V(G_i))$ then by suitably modifying the maps in $\mathcal{I}$ without affecting the graph, we can reduce the value of $m$ by one for each such $t$. This is straightforward, and we discuss another way of reducing the value of $m$ below.

\textbf{Situation 2:} For $i \in [m],$ Let $F_i = \{v \in V(G_i)$ : $I_i(v)\neq I_j(u)$ for any $u \in V(G_j), j \in [m]\backslash \{i\}\}.$   Let $F = \cup_{i=1}^k F_i$ and $l(F) =\{c \in [m] : c=I_i(v)  \ \text{for some}\ v \in F_i\}$. 

 Now, for a fixed $c\in l(F),$ for each $i\in [k],$ let $E'_i$  be the matrix obtained from the matrix $E
_i$  (c.f. Definition \ref{def E}) by deleting its $c$-th column and hence $E'_i$ is an $n_i\times m-1$ matrix.  Given these notions, we have

\begin{lem}[$Reduction \ Lemma$]\label{reduction}
The adjacency matrix $A(G)$ of the graph $G$ satisfies
 \begin{equation}\label{reduced}
 A(G)=
\begin{bmatrix}
    A(G_1) &  \rho_{1,2}E_1E_2^t & \cdots  & \rho_{1,k} E_1E_k^t \\
			\rho_{2,1} E_2 E_1^t & A(G_2) & \cdots &  \rho_{2,k} E_2E_k^t  \\
			\vdots  & \vdots  & \ddots   & \vdots\\
			\rho_{k,1} E_kE_1^t &\rho_{k,2} E_kE_2^t & \cdots \cdots  &A(G_k)
\end{bmatrix}=\begin{bmatrix}
    A(G_1) & \rho_{1,2} E_1^{'}E_2^{'t} & \cdots  &  \rho_{1,k}E_1^{'} E_k^{'t} \\
			\rho_{2,1} E_2^{'} E_1^{'t} & A(G_2) & \cdots &  \rho_{2,k}E_2^{'}E_k^{'t}  \\
			\vdots  & \vdots  & \ddots   & \vdots\\
			\rho_{k,1} E_k^{'} E_1^{'t} & \rho_{k,2} E_k^{'} E_2^{'t} & \cdots \cdots  &A(G_k)
\end{bmatrix}.
\end{equation}

In particular, the graph $G$ can be expressed as an $H_{m-1}$-join.
\end{lem}

\begin{pf}
\begin{align*}
    (E_iE_j^t)_{rs} &= \begin{cases}
    1& \text{if}\ I_i(v_i^r)=I_j(v_j^s)\\
    0 & \text{otherwise.}
\end{cases}\\
&= \begin{cases}
    1& \text{if}\ v_i^r \notin F_i, v_j^s \notin F_j \ \text{and} \ (E_iE_j^t)_{rs} =1 \\
    0 & \text{otherwise.}
\end{cases}\\
&= (E_i^{'}E_j^{'t})_{rs}.
\end{align*}
\end{pf}

By the above lemma, for each $c \in l(F) $, the value of $m$ can be decreased by one.

By repeatedly applying this process on the labels in $l(F)$, we see that $G$ can be expressed as an $H_{m_d}$-join, where $m_d=\lvert l(F)\rvert$.

\textbf{Situation 3:}
Along the same lines, 
Let $F'_i = \{v \in V(G_i)$ : $I_i(v)\neq I_j(u)$ for any $u \in V(G_j) \ \text{with}\ v_j \in N_H(v_i)\}.$   Let $F^{'}=\{c \in [m] : \text{whenever}\ c=I_i(v)\ \text{for some}\ i \in [k],  v \in F'_i \}$. Letting $m^{'}_d=\lvert F^{'}\rvert$, we see that $G$ can be equivalently expressed as an $H_{m^{'}_d}$-join as well. Note that $m \geq m_d \geq m_d^{'}$. And we get
\begin{cor}
    Let the notations be as in Theorem \ref{mainn}. then 
     If $\lambda$ is an $E_i$-main eigenvalue of $A(G_i)$, then $\lambda$ is an eigenvalue of $A(G)$ with multiplicity at least $\mult_i(\lambda)-m_d^{'}$.
\end{cor}
\subsection{\textit{Universal Spectra of $H_m$-joins}}
This subsection studies the universal spectra of graphs obtained from the $H_m$-join operation. The universal adjacency matrix of a graph $G$ is defined to be $U(G):=\alpha A(G) + \beta I + \gamma J + \delta D(G)$ where $\alpha, \beta, \gamma, \delta \in \mathbb R$.  Universal adjacency matrix generalizes many interesting matrices associated with graph $G$ such as Laplacian $(\alpha,\beta,\gamma,\delta)=(1,0,0,-1)$, signless Laplacian $(\alpha,\beta,\gamma,\delta)=(1,0,0,1)$, $A_{\alpha}$ matrix $(\alpha,\beta,\gamma,\delta)=(1-\alpha,0,0,\alpha)$$, \alpha \in [0,1]$, and Seidal matrix $(\alpha,\beta,\gamma,\delta)=(-2,1,1,0)$ thus gaining its importance. 

Let the notations be as in Theorem \ref{mainn}. From Equation (\ref{AoomfG}), we get 
 
\begin{equation}\label{unis}
    U(G) = \begin{bmatrix}
         U(G_1)+\delta \mathcal D_1 &  \rho_{1,2} E_1E_2^t+\gamma\bold J_{n_1 \times n_2} & \cdots  & \rho_{1,k} E_1E_k^t+\gamma \bold J_{n_1 \times n_k} \\
			\rho_{2,1} E_2 E_1^t+\gamma \bold J_{n_2 \times n_1} & U(G_2)+\delta \mathcal D_2 & \cdots & \rho_{2,k} E_2E_k^t+\gamma \bold J_{n_2 \times n_k}  \\
			\vdots  & \vdots  & \ddots   & \vdots\\
			\rho_{k,1} E_kE_1^t+\gamma \bold J_{n_k \times n_1} & \rho_{k,2} E_kE_2^t+\gamma \bold J_{n_k \times n_2} & \cdots \cdots  &U(G_k)+\delta \mathcal D_k
    \end{bmatrix}
\end{equation}
where for $i\in [k],$ $(\mathcal D_i)_{st}=\begin{cases}
    \sum_{v_i v_j \in E(H)} \lvert \{v_j^p \in V(G_j): I_j(v_j^p) = I_i(v_i^s)\}\rvert & \text{if} \ s=t,\\
    0 &  \text{otherwise}.   
\end{cases}$

Since the same proof (as in Theorem \ref{mainn}) goes through if we let $\gamma=0$ in the above equation and let $U(G_i)+\delta \mathcal D_i$ take the role of $A(G_i)$ for each $i \in [k]$, we get 
\begin{equation}
			\det(\lambda I_n - U(G)) = \Bigg( \prod_{i=1}^k \phi_{U(G_i)+\delta \mathcal D_i}(\lambda) \Bigg) \det(\widetilde{U}) .
		\end{equation}  where
 $$
		\widetilde{U} :=  \begin{bmatrix}
			I_m & -\rho_{1,2}\Gamma_{(U(G_1)+\delta \mathcal D_1)}(E_1)  & \cdots & -\rho_{1,k}\Gamma_{(U(G_1)+\delta \mathcal D_1)}(E_1) \\
			-\rho_{2,1}\Gamma_{(U(G_2)+\delta \mathcal D_2)}(E_2)  & I_m & \cdots & -\rho_{2,k}\Gamma_{(U(G_2)+\delta \mathcal D_2)}(E_2)\\
			\vdots & \vdots & \ddots & \vdots  \\
			-\rho_{k,1}\Gamma_{(U(G_k)+\delta \mathcal D_k)}(E_k) & -\rho_{k,2}\Gamma_{(U(G_k)+\delta \mathcal D_k)}(E_k) & \cdots &I_m
		\end{bmatrix}.$$ 
  where for $i \in [k],$ $$ \Gamma_{(U(G_i)+\delta \mathcal D_i)}(E_i)=E_i^{t} (\lambda I_{n_i}-(U(G_i)+\delta \mathcal D_i))^{-1} E_i.$$
  and 
  $$\phi_{(U(G_i)+\delta \mathcal{D}_i)}(\lambda)=\det(\lambda I_{n_i}-(U(G_i)+\delta \mathcal{D}_i)).$$
Thus, we have the following corollary. 
 \begin{cor}
     Let the notations be as in Theorem \ref{mainn}. Let $\lambda$ be an eigenvalue of the matrix $\alpha A(G_i) + \beta I_{n_i} + \delta (D(G_i)+\mathcal D_i)$with multiplicity $\mult_i(\lambda)$. Then, we have the following. 
     \begin{itemize}
         \item If $\lambda$ is an $E_i$-non-main eigenvalue of $\alpha A(G_i) + \beta I_{n_i} + \delta (D(G_i)+\mathcal D_i)$, then $\lambda$ is an eigenvalue of $\alpha A(G) + \beta I + \delta D(G) \ \ $ with multiplicity $\mult_i(\lambda)$.
         \item If $\lambda$ is an $E_i$- main eigenvalue of $\alpha A(G_i) + \beta I_{n_i} + \delta (D(G_i)+\mathcal D_i)$, then $\lambda$ is an eigenvalue of $\alpha A(G) + \beta I + \delta D(G)$ with multiplicity at least $\mult_i(\lambda)-m$.
         \item Remaining eigenvalues are roots of the polynomial $\Phi^{'}(\lambda) \in \mathbb{C}[\lambda].$ \end{itemize}
 where $\Phi_1(\lambda)$ is analogously defined as in Equation (\ref{4.8}).

  In particular, we conclude the following about the Laplacian matrix $L(G)$ and Signless Laplacian matrix $Q(G)$, respectively.
  
   Let $\lambda$ be an eigenvalue of the matrix $L(G_i)-\mathcal D_i$ with multiplicity $\mult_i(\lambda)$. Then, we have the following.
    \begin{itemize}
\item If $\lambda$ is an $E_i$-non-main eigenvalue of $L(G_i)-\mathcal D_i$, then $\lambda$ is an eigenvalue of $L(G)$ with multiplicity at least $\mult_i(\lambda)$.
\item If $\lambda$ is an $E_i$-main eigenvalue of $L(G_i)-\mathcal D_i$, then $\lambda$ is an eigenvalue of $L(G)$ with multiplicity at least $\mult_i(\lambda)-m$.
\item Remaining eigenvalues are the roots of the polynomial $\Phi^{'}(\lambda) \in \mathbb{C}[\lambda].$
\end{itemize} where $\Phi_2(\lambda)$ is analogously defined as in Equation (\ref{4.8}). 

Similarly, 
Let $\lambda$ be an eigenvalue of the matrix $L(G_i)+\mathcal D_i$ with multiplicity $\mult_i(\lambda)$. Then, we have the following.
    \begin{itemize}
\item If $\lambda$ is an $E_i$-non-main eigenvalue of $L(G_i)+\mathcal D_i$, then $\lambda$ is an eigenvalue of $Q(G)$ with multiplicity at least $\mult_i(\lambda)$.
\item If $\lambda$ is an $E_i$-main eigenvalue of $L(G_i)+\mathcal D_i$, then $\lambda$ is an eigenvalue of $Q(G)$ with multiplicity at least $\mult_i(\lambda)-m$.
\item Remaining eigenvalues are the roots of the polynomial $\Phi^{'}(\lambda) \in \mathbb{C}[\lambda].$
\end{itemize} where $\Phi_3(\lambda)$ is again analogously defined as in Equation (\ref{4.8}). 
\end{cor}

\section{Applications to $H$-generalized Join of Graphs}\label{hgeneralizedjoin}

\begin{defn}\cite{cons}\textit{H-generalized join.}
   Let $H$ be a graph on the vertex set $V(H) =\{v_1,v_2,\dots,v_k\}$. Let $\mathcal{F} = \{G_i \colon 1\leq i \leq k\}$ be an arbitrary family of graphs with $V(G_i) = \{v_i^1,\dots,v_i^{n_i}\}$ and let $\mathcal{S}=\{S_1,S_2,\dots,S_k\} \ \text{with} \  S_i \subseteq V(G_i)$ for $1\leq i \leq k$.  The {\em{H}-generalized join} of the family $\mathcal{F}$ of graphs constrained by $\mathcal{S}$, denoted by $\bigvee_{H,\mathcal{S}}^{\mathcal{F}}$, is obtained by replacing each vertex $v_i$ of $H$ by the graph $G_i$ in $H$ such that if $v_i$ is adjacent to $v_j$ in $H$ then each vertex of $S_i$ is adjacent to every vertex of $S_j$ in $\bigvee_{H,\mathcal{S}}^{\mathcal{F}}$. More precisely, 
 \begin{enumerate}
     \item $V\big( \bigvee_{H,\mathcal{S}}^\mathcal{F} \big) = \displaystyle \bigsqcup_{i=1}^k V(G_i)$, and
     \item $E\big( \bigvee_{H,\mathcal{S}}^\mathcal{F} \big) = \big(\displaystyle  \bigsqcup_{i=1}^k E(G_i)\big) \sqcup \big(   \bigsqcup_{ (v_i,v_j) \in E(H)} \{ xy : x \in S_i, y \in S_j \}\big).$
 \end{enumerate}
\end{defn}


\subsection{$H$-generalized join as $H_m$-join } Let $G$ be the $H$-generalized join of the family of graphs  $\mathcal{F}=\{G_1,G_2,\dots,G_k\}$ with respect to $\mathcal{S}=\{S_1,S_2,\dots,S_k\}$ where $S_i \subseteq V(G_i).$ We can write $G$ as $H_m$-join of $G_1,G_2,\dots,G_k$ for $m=k+1$ as follows. 

For each $1 \le i \le k$, we define the indexing map $I_i$ as follows.
\begin{equation}\label{constohm}
    I_i(v_i^s)=\begin{cases}
    1 & \text{if} \ v_i^s\in S_i\\
    i+1 & \text{otherwise}
\end{cases}
\end{equation}

Then we have the $H$-generalized join  $\bigvee_{H,\mathcal{S}} \mathcal{F}$ is the same as the $H_m$-join $\bigvee_{H}^{ \mathcal{F}, \mathcal I}$ where $\mathcal I = \{I_i: 1 \le i \le k\}$. For example,

\begin{example}
  Clearly, the following graph is a $(P_4)$-generalized join of $K_3,P_4,C_5,K_{3,3}$. Note that, this is also the $(P_4)_5$-join of $K_3,P_4,C_5,K_{3,3}$ for the given vertex labelling. 

\begin{figure}[h]
\centering
\begin{tikzpicture}[
every edge/.style = {draw=black,very thick},
 vrtx/.style args = {#1/#2}{%
      circle, draw, thick, fill=black,
      minimum size, label=#1:#2}
                    ]
\node(A) [vrtx=left/1] at (0, 1.5) {};
\node(B) [vrtx=left/2] at (-1, 0.5) {};
\node(C) [vrtx=left/2] at (0,-.5) {};
\path   (A) edge (B)
        (A) edge (C)
        (B) edge (C);

\node (2A) [vrtx=left/3]     at ( 2.5, 0) {};
\node (2B) [vrtx=left/3]     at (2.5, 1) {};
\node (2C) [vrtx=above/1]    at ( 2.5,2) {};
\node (2D) [vrtx=left/1]    at ( 2.5,-1) {};
\path   (2A) edge (2B)
        (2A) edge (2C)
        (2A) edge (2D);

\path (A) edge[blue] (2C)
(A) edge[blue] (2D);

\node (3A) [vrtx=above/1]     at ( 6.5, 2) {};
\node (3B) [vrtx=right/4]     at (5, 0.5) {};
\node (3C) [vrtx=left/1]    at ( 8,0.5) {};
\node (3D) [vrtx=right/4]    at ( 7.5,-1) {};
\node (3E) [vrtx=left/1]    at ( 5.5,-1) {};
\path   (3A) edge (3B) 
(3B) edge (3E)
(3D) edge (3C)
(3E) edge (3D)
(3C) edge (3A)
(2C) edge[blue] (3A)
(2C) edge[blue] (3C)
(2C) edge[blue] (3E)
(2D) edge[blue] (3C)
   (2D) edge[blue] (3A)
      (2D) edge[blue] (3E);
    
\node (4A) [vrtx=above/5]     at ( 10.5, 1.5) {};
\node (4B) [vrtx=above/5]     at (12, 1.5) {};
\node (4C) [vrtx=above/1]    at ( 13.5,1.5) {};
\node (4D) [vrtx=below/5]    at ( 10.5,-.5) {};
\node (4E) [vrtx=below/5]    at ( 12,-.5) {};
\node (4F) [vrtx=below/1]    at ( 13.5,-.5) {};

\path (4A) edge  (4D) edge (4E) edge (4F)
(4B) edge  (4D) edge (4E) edge (4F)
(4C) edge  (4D) edge (4E) edge (4F)

(3A) edge[blue,bend left] (4C) edge[blue] (4F)
(3C) edge[blue] (4C) edge[blue] (4F)
(3E) edge[blue] (4C) edge[blue,bend right] (4F)
;
\end{tikzpicture}
\caption{$(P_4)_5$- generalized join of $K_3,P_4,C_5,K_{3,3}$}
    \label{fig:figure-4}
\end{figure}
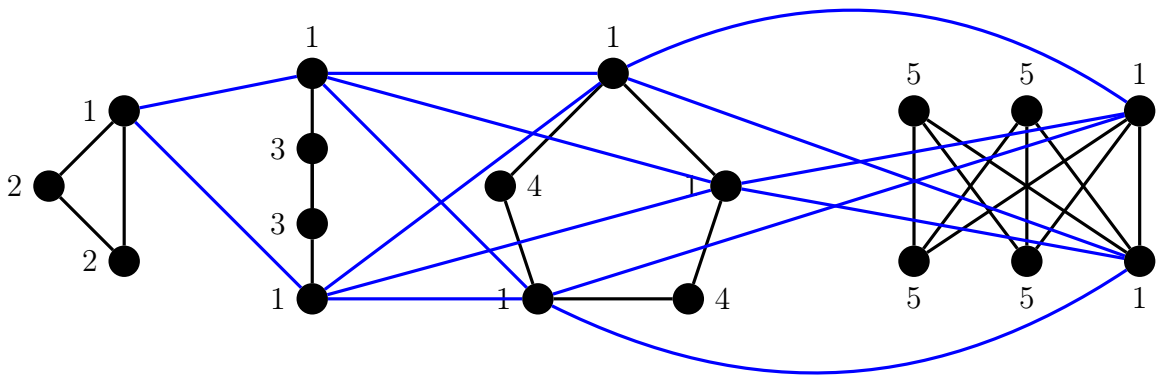

\end{example}

 A $(k,\tau)$-regular set in a graph $G$ is a subset $S$ $\subseteq$ $V(G)$  which induces a $k$-regular subgraph in $G$ such that every vertex outside $S$ has $\tau$ neighbors in $S$ \cite{cdrm}. 

The main result proven in \cite{cons} is the following.
\begin{thm}\cite[Theorem 1]{cons}
Consider a graph $H$ of order $k$ and a family of graphs $\mathcal{F}=\{ G_1,G_2,\cdots,G_k\}$ such that $ V(G_i) = \{v_i^{1},\dots,v_i^{n_i}\}$  for $1 \le i \le k$. Consider also the family of vertex subsets $\mathcal{S}=\{S_1,S_2,\cdots,S_k\},$ where 
$$ S_i \in \{ S_i^{'} \subseteq V(G_i) : \text{either}\ S_i^{'}\  \text{or}\ V(G_i)\setminus S_i^{'}\ \text{is}\ (k_i,\tau_i)-\text{regular in $G_i$ for some integers }\ k_i,\tau_i \}, $$ \text{for}\ $i=1,\dots,k. $
Let $G$ be the $H$-generalized join $\bigvee_{H,\mathcal{S}}^\mathcal{F}$. If $\lambda \ \ (\neq k_i-\tau_i) \in \sigma(G_i)$  is a non-main eigenvalue, then $\lambda \in \ \sigma(G).$ 
\end{thm}

We note that $(k_i,\tau_i)$ regularity conditions are assumed on the sets $S_i$ to prove that non-main eigenvalues of $G_i$ (other than $k_i-\tau_i$) get carried forward to $G$. We prove the same without assuming any regularity conditions on the sets $S_i$, as a corollary of our Theorem \ref{mainn} using the following \cite[Lemma 1]{cons}. 

\begin{lem} \cite[Lemma 1]{cons}
    Let $G$ be a graph with a $(k,\tau)$-regular set $S$, where $\tau>0,$ and $\lambda \in \sigma(A(G)).$ Then $\lambda$ is non-main if and only if $$\lambda = k-\tau \ \text{or} \ \bold 1_{S} \in (\xi_{G} (\lambda))^{\perp}.$$
\end{lem}
\begin{cor}
    
Consider a graph $H$ of order $k$ and a family of  graphs $\mathcal{F}=\{ G_1,G_2,\cdots,G_k\}$ such that $\lvert V(G_i) \rvert =n_i,  \ \text{for}\ i=1,\dots,k.$ Consider also the family of vertex subsets $\mathcal{S}=\{S_1,S_2,\cdots,S_k\}$ where $S_i\subseteq V(G_i)$  for $i=1,2,\dots,k$.  Let $\lambda$ be an eigenvalue of $G_i$ for some $i$. Let $G$ be the $H$-generalized join of $G_1,G_2,\dots,G_k$ with respect to $S_1,S_2,\dots,S_k$. Then
\text{if} $\bold 1_{S_i} \in (\xi_{G_i}(\lambda))^{\perp}, \text{then }  \lambda \in \sigma(G).$
\begin{pf}
    First we write $G$ as an $H_{k+1}$-join of $\mathcal{F}$ through the maps given in Equation (\ref{constohm}).
   Now, from repeated application of Lemma \ref{reduction}, we get
\begin{equation}\label{reducedcons}
    A(G)=
\begin{bmatrix}
    A(G_1) &  E_1E_2^t & \cdots  &  E_1E_k^t \\
			 E_2 E_1^t & A(G_2) & \cdots &  E_2E_k^t  \\
			\vdots  & \vdots  & \ddots   & \vdots\\
			 E_kE_1^t & E_kE_2^t & \cdots \cdots  &A(G_k)
\end{bmatrix}=\begin{bmatrix}
    A(G_1) &  \bold 1_{S_1} \bold 1_{S_2}^t & \cdots  &   \bold 1_{S_1} \bold 1_{S_k}^t \\
			\bold 1_{S_2} \bold 1_{S_1}^t & A(G_2) & \cdots &   \bold 1_{S_2} \bold 1_{S_k}^t  \\
			\vdots  & \vdots  & \ddots   & \vdots\\
			 \bold 1_{S_k}\bold 1_{S_1}^t & \bold 1_{S_k}\bold 1_{S_2}^t & \cdots \cdots  &A(G_k)
\end{bmatrix}\end{equation}

    and the statement follows from Theorem \ref{mainnn}.
\end{pf}
\end{cor}

\subsection{\textit{Universal spectrum of $H$-generalized join of graphs}}

Consider a graph $H$ on the vertex set $V(H) =\{v_1,v_2,\dots,v_k\}$ and a family of graphs $\mathcal{F}=\{ G_1,G_2,\cdots,G_k\}$ such that $V(G_i)  =\{v_i^1,v_i^2,\dots,v_i^{n_i}\},  \ \text{for}\ i=1,\dots,k.$ Consider also the family of vertex subsets $\mathcal{S}=\{S_1,S_2,\cdots,S_k\},$ where $S_i \subseteq V(G_i).$ 
Let $G$ be the $H$-generalized join of $\mathcal{F}$ with respect to $\mathcal{S}$. First we observe that for $j\in [n_i],$
\begin{equation}\label{obscons}
\text{deg}_{G}(v_i^j)=\text{deg}_{G_i}(v_i^j)+w_i, \ \text{where } w_i=\sum_{v_iv_l \in E(H)} \lvert S_l \rvert.\end{equation} 
Then by Equations \eqref{reducedcons} and (\ref{obscons}),
$$U(G)=\begin{bmatrix}
    U(G_1)+\delta w_1 I_{n_1} & \bold 1_{S_1}\bold 1_{S_2}^t + \gamma \bold 1_{n_1} \bold 1_{n_2}^t  & \cdots  &   \bold 1_{S_1} \bold 1_{S_k}^t + \gamma  \bold 1_{n_1} \bold 1_{n_k}^t \\
			\bold  1_{S_2} \bold 1_{S_1}^t + \gamma \bold 1_{n_2}\bold 1_{n_1}^t &U(G_2)+\delta w_2 I_{n_2} & \cdots &  \bold  1_{S_2} \bold 1_{S_k}^t +\gamma \bold  1_{n_2}\bold 1_{n_k}^t \\
			\vdots  & \vdots  & \ddots   & \vdots\\
			\bold  1_{S_k}\bold 1_{S_1}^t + \gamma  \bold 1_{n_k} \bold 1_{n_1}^t &  \bold 1_{S_k} \bold 1_{S_2}^t +\gamma  \bold 1_{n_k}\bold 1_{n_2}^t & \cdots \cdots  &U(G_k)+\delta w_k I_{n_k}
\end{bmatrix}.$$
We order the vertices in $V(G_i)$ so that the vertices in $S_i$ are put first, we see that 
$$\bold 1_{S_i}\bold 1_{S_j}^t + \gamma \bold 1_{n_i}\bold 
1_{n_j}^t = L_i L_j^t,$$ where $L_i$ is an $n_i \times 2$ matrix defined by
$(L_i)_{st}=\begin{cases}
    \sqrt{\gamma } & \text{if}\ t=1 \\
     1 & \text{if} \ t=2\ \text{and}\ v_i^s \in S_i \\
    0 & \text{otherwise}.
\end{cases}$

Thus, by Theorem \ref{mainn},
\begin{eqnarray}
\begin{aligned}
    \label{Gamma}
    \det(\lambda I_n - U(G))&=\bigg(\prod_{i=1}^{k-1} \phi_{U(G_i)+\delta w_i I_{n_i}}(\lambda)\bigg)\\
	&\indent \times	\begin{vmatrix}
			I_2 & -{\Gamma_{(U(G_1)+\delta w_1 I_{n_1})}(L_1)}  & \cdots & -\Gamma_{(U(G_1)+\delta w_1 I_{n_1})}(L_1)  \\
			-\Gamma_{(U(G_2)+\delta w_2 I_{n_2})}(L_2)  & I_2 & \cdots & -\Gamma_{(U(G_2)+\delta w_2 I_{n_2})}(L_2)\\
			\vdots & \vdots & \ddots & \vdots  \\
			-\Gamma_{(U(G_k)+\delta w_k I_{n_k})}(L_k) & -\Gamma_{(U(G_k)+\delta w_k I_{n_k})}(L_k) & \cdots & I_2
		\end{vmatrix},
\end{aligned}
\end{eqnarray}
where \begin{eqnarray*}
    \begin{aligned}
    \Gamma_{(U(G_i)+\delta w_i I_{n_i})}(L_i)&=L_i^t(\lambda I_{n_i}-(U(G_i)+\delta w_i I_{n_i}))^{-1}L_i, \\
&=\begin{bmatrix}
 \gamma \bold 1_{n_i}^t(\lambda I_{n_i}-(U(G_i)+\delta w_i I_{n_i}))^{-1} \bold
1_{n_i} & \sqrt{\gamma} \bold 1_{n_i}^t(\lambda I_{n_i}-(U(G_i)+\delta w_i I_{n_i}))^{-1} \bold 1_{S_i} \\ \sqrt{\gamma}\bold 1_{S_i}^t
(\lambda I_{n_i}-(U(G_i)+\delta w_i I_{n_i}))^{-1} \bold 1_{n_i}
& \bold 1_{S_i}^t(\lambda I_{n_i}-(U(G_i)+\delta w_i I_{n_i}))^{-1}\bold 1_{S_i}
\end{bmatrix}.
\end{aligned}
\end{eqnarray*}
\subsection{Construction of co-spectral graphs}

Equation \eqref{Gamma} above expresses the characteristic polynomial of the universal adjacency matrix $U(G)$ for an arbitrary $H$-generalized join $G$ of graphs. We will use this expression to construct infinite families of cospectral graphs. First, we will find the values of the entries in the main function $L_i^t(\lambda I_{n_i}-(U(G_i)+\delta w_i I_{n_i}))^{-1}L_i$ in some particular cases.

\begin{lem}\label{maincorollary}
    Suppose $G$ is a graph on $n$ vertices and $S, S_1, S_2 \subseteq V(G)$. then
\begin{enumerate} 
 \item If $G$ is an $r$-regular graph and $\delta=0,$ then $$\bold  1_{S}^t (\lambda I_n-U(G))^{-1} \bold 1_{n} = \bold 1_{n}^t (\lambda I_n-U(G))^{-1}\bold 1_{S} = \frac{\lvert S \rvert}{\lambda-(\alpha r+\beta+\gamma n)}.$$
    \item If $S_1=V(G)$ and $\alpha=-\delta,$
     $$\bold 1_{S_2}^t(\lambda I_n-U(G))^{-1}\bold 1_{S_1}=\frac{\lvert S_2 \rvert}{\lambda-( \beta +\gamma n )}.$$
     \item  If $S_2=V(G)$ and $\alpha=-\delta,$
     $$\bold 1_{S_2}^t(\lambda I_n-U(G))^{-1}\bold 
     1_{S_1}=\frac{\lvert S_1 \rvert}{\lambda-(\beta +\gamma n)}.$$
\end{enumerate}

\end{lem}
\begin{proof}
Suppose $G$ is an $r$-regular graph and $\delta=0,$ 
Then
    $$(\lambda I_n-U(G))\bold 1_{n}= (\lambda I_n-(\alpha A(G)+ \beta I + \gamma J)\bold 1_{n}= (\lambda-(\alpha r+ \beta + \gamma n
    )) \bold 1_{n}$$ which implies $\bold 1_{S}^t \bold 1_{n}=\bold 1_{S}^t (\lambda I_n-U(G\bold ))^{-1} (\lambda-(\alpha r+\beta+\gamma n ))\bold 1_{n}$ and so
    $$\bold 1_{S}^t(\lambda I_n-U(G))^{-1}\bold 1_{n}=\frac{\lvert S \rvert}{(\lambda-(\alpha r+\beta+\gamma n ))}.$$
    Similarly we have 
    $$\bold 1_{n}^\bold t(\lambda I_n-U(G))^{-1}\bold 1_{S}=\frac{\lvert S \rvert}{(\lambda-(\alpha r+\beta+\gamma n))}.$$
    
    Now if $\alpha=-\delta$ and $S_1=V(G),$
    $$(\lambda I_n-U(G))\bold 1_{S_1}=(\lambda-(\beta +  \gamma n)) \bold 1_{S_1}$$ which implies $\bold 1_{S_2}^t \bold 1_{S_1}=(\lambda-( \beta+\gamma n))\bold 1_{S_2}^t(\lambda I_n-U(G))^{-1}\bold 1_{S_1}$ and so
    $$\bold 1_{S_2}^t(\lambda I_n-U(G))^{-1}\bold 1_{S_1}=\frac{\lvert S_2 \rvert}{(\lambda-(\beta+\gamma n ))}.$$
    (3) follows similarly.
\end{proof}
Using Equation \eqref{Gamma} and Lemma \ref{maincorollary}, we now find infinite families of cospectral graphs, which are realized as $H_m$-joins of graphs. In particular, we find these families as $H$-generalized joins of graphs.

\begin{thm}\label{cospcor}
    Let $H$ be a graph on $k$ vertices. Let $G_i$ and  $G_i^{'}$ be graphs on $n_i$ vertices for $i\in [k]$. Let $S_i \subseteq V(G_i)$, $S_i^{'}\subseteq V(G_i^{'})$ with $\lvert S_i \rvert=\lvert S_i^{'} \rvert$. 
    \begin{enumerate}

        \item If, for $i\in [k],$ $G_i$ and $G_i^{'}$ are $r_i$-regular and $A$-cospectral, $\Gamma_{A(G_i)}(\bold 1_{S_i})=\Gamma_{A(G_i^{'})}(\bold 1_{S^{'}_i})$, then $\bigvee_{H,\mathcal{S}} \mathcal{F}$ and   $\bigvee_{H,\mathcal{S'}} \mathcal{F'}$ are $A$-cospectral.
        \item If, for $i\in [k],$ $G_i$ and $G_i^{'}$ are $r_i$-regular and $S$-cospectral, $\Gamma_{S(G_i)}(\bold 1_{S_i})=\Gamma_{S(G_i^{'})}(\bold 1_{S^{'}_i})$, then $\bigvee_{H,\mathcal{S}} \mathcal{F}$ and   $\bigvee_{H,\mathcal{S'}} \mathcal{F'}$ are $S$-cospectral.
        \item If, for $i\in [k],$ $G_i$ and $G_i^{'}$ are $L$-cospectral, $\Gamma_{L(G_i)}(\bold 1_{S_i})=\Gamma_{L(G_i^{'})}(\bold 1_{S^{'}_i})$ then $\bigvee_{H,\mathcal{S}} \mathcal{F}$ and   $\bigvee_{H,\mathcal{S'}} \mathcal{F'}$ are $L $-cospectral.
    \item If, for each $i\in[k]$, $G_i$ and $G_i^{'}$ are $U$-cospectral and ${\Gamma_{U(G_i)+\delta w_i I_{n_i}}(L_i)}={\Gamma_{U(G_i^{'})+\delta w_i I_{n_i}}(L_i^{'})}$ (as defined in Equation \eqref{Gamma}), then $\bigvee_{H,\mathcal{S}} \mathcal{F}$ and   $\bigvee_{H,\mathcal{S'}} \mathcal{F'}$ are $U$-cospectral.
    \end{enumerate}
\end{thm}
\begin{pf}
 To prove (1), we observe that if $G_i$ and $G_i^{'}$ are $r_i$-regular, by Lemma \cite[Proposition 6]{leman}, $$\bold 1_{n_i}^t(\lambda I_{n_i}-A(G_i))^{-1} \bold
1_{n_i}=(\Gamma_{A(G_i)}(L_i))_{11}=\frac{n_i}{\lambda-r_i}=(\Gamma_{A(G^{'}_i)}(L_i^{'}))_{11}= \bold 1_{n_i}^t(\lambda I_{n_i}-A(G_i^{'}))^{-1}\bold 
1_{n_i}.$$ Since $\lvert S_i \rvert =\lvert S_i^{'} \rvert$, by Lemma \ref{maincorollary}(1), $$(\Gamma_{A(G_i)}(L_i))_{12}=(\Gamma_{A(G_i)}(L_i))_{21}=\frac{\lvert S_i \rvert}{\lambda-r_i}.$$
$$(\Gamma_{A(G_i^{'})}(L_i^{'}))_{12}=(\Gamma_{A(G_i^{'})}(L_i^{'}))_{21}=\frac{\lvert S_i \rvert}{\lambda-r_i}.$$
Also,  $(\Gamma_{A(G_i)}(L_i))_{22}=\Gamma_{A(G_i)}(\bold 1_{S_i})=\Gamma_{A(G_i^{'})}(\bold 1_{S^{'}_i})=(\Gamma_{A(G_i^{'})}(L_i^{'}))_{22}$.

Thus (1) follows from Equation \eqref{Gamma}. Similarly, we have (2). To prove (3), we observe that by Lemma \ref{maincorollary}(2),\ref{maincorollary}(3),
$$ \bold 1_{n_i}^t(\lambda I_{n_i}-L(G_i))^{-1} \bold
1_{n_i}=((\Gamma_{L(G_i)}(L_i))_{11}=\frac{n_i}{\lambda}=(\Gamma_{L(G_i^{'})}(L_i^{'}))_{11}=\bold 1_{n_i}^t(\lambda I_{n_i}-L(G_i^{'}))^{-1}\bold
1_{n_i},$$ 
$$(\Gamma_{L(G_i)}(L_i))_{12}=(\Gamma_{L(G_i)}(L_i))_{21}=\frac{\lvert S_i \rvert}{\lambda}.$$
$$((\Gamma_{L(G_i^{'})}(L_i^{'}))_{12}=(\Gamma_{L(G_i^{'})}(L_i^{'}))_{21}=\frac{\lvert S_i \rvert}{\lambda}.$$
Also,  $(\Gamma_{L(G_i)}(L_i))_{22}=\Gamma_{L(G_i)}(\bold 1_{S_i})=\Gamma_{L(G_i^{'})}(\bold 1_{S^{'}_i})=(\Gamma_{L(G_i^{'})}(L_i^{'}))_{22}$, which implies (3) by Equation \eqref{Gamma}. Claim (4) follows immediately from Equation \eqref{Gamma}.
\end{pf}

\begin{rem}
    We note that starting with two (regular in the cases of (1) and (2) above) cospectral graphs, we can construct bigger cospectral graphs which are $H$-generalized joins of graphs given that their respective main functions $\Gamma_{X(G_i)}(\bold 1_{S_i})$ and  $\Gamma_{X(G_i^{'})}(\bold 1_{S_i^{'}})$ match. These main functions can be explicitly calculated using techniques similar to those in \cite{leman}.
\end{rem}

\subsection{Future Directions:}
\begin{enumerate}
      \item In this paper, we have shown that the graphs obtained as the Cartesian product of two graphs, $H$-join, and the $H$-generalized join of a family of graphs can be realized as $H_m$-joins. Using this realization and the results of this paper, we have studied their spectrum. Find other operations that are particular cases of the $H_m$-join operation and study the spectrum of the resulting graphs using the results proved in this paper.

    \item Characterise the graphs which can be realized as $H_m$-join for a suitable choice of graphs for each $m \in \mathbb N$. This problem is partially discussed in this paper. For example, the case $m=1$.

    \item Generalize the results known for main and non-eigenvalues of graphs from the literature to the more general $E$-main eigenvalues defined in this paper. For example, we can study when a graph $G$ has exactly $k$ number of $E$-main eigenvalues for a matrix $E$ \cite{du}.

\end{enumerate}

\noindent In part II of this paper, which is under preparation, we are working on providing complete/partial answers to the above questions and more.

\end{document}